\documentclass[10pt]{article} 
\usepackage{geometry}
\usepackage{parskip}
\usepackage{amssymb, amsmath, amsthm}
\usepackage{amsfonts}
\numberwithin{equation}{section}
\usepackage[colorlinks=true,linkcolor=blue,anchorcolor=blue,citecolor=black,urlcolor=blue,plainpages=false,pdfpagelabels]{hyperref}

\usepackage[utf8]{inputenc} 
\usepackage[T1]{fontenc}    
\usepackage{hyperref}       
\usepackage{url}            
\usepackage{booktabs}       
\usepackage{amsfonts}       
\usepackage{nicefrac}       
\usepackage{microtype}      
\usepackage{graphicx}
\usepackage{natbib}
\usepackage{doi}
\usepackage{dsfont}
\usepackage{tikz-cd} 
\usepackage{amscd} 
\usepackage{physics}
\usepackage{float}
\usepackage{subcaption}
\usepackage{tabularray}
\usepackage{multirow}

\newtheorem{theorem}{Theorem}
\newtheorem{corollary}[theorem]{Corollary}

\newtheorem{proposition}[theorem]{Proposition}

\theoremstyle{definition}
\newtheorem{definition}[theorem]{Definition}
\theoremstyle{remark}
\newtheorem{remark}[theorem]{Remark}

\newcommand{\R}{\mathbb{R}}

\newcommand{\Ibb}{\mathbb{I}}
\newcommand{\Fbb}{\mathbb{F}}

\newcommand{\Zbb}{\mathbb{Z}}
\renewcommand{\P}{\mathsf{P}}

\renewcommand{\Im}{\mathrm{Im} \,}

\renewcommand{\H}{\mathrm{H}}

\renewcommand{\rank}{\mathrm{rank} \,}

\newcommand{\Dgm}{\mathrm{Dgm} \,}
\newcommand{\Barc}{\mathrm{Bar} \,}

\newcommand{\Vect}{\mathsf{Vec}}

\def\ttcond{%
\ttfamily\fontseries{mc}\fontshape{n}\selectfont%
}
\def\url@dfstyle{%
\def\UrlFont{\ttcond}%
\@ifundefined{Url@ttdo}{}{\Url@ttdo}%
}
\DeclareUrlCommand\webLink{\urlstyle{df}}
\let\webLinkFont=\ttcond

\def\httpsAddr#1{\href{https:#1}{\webLinkFont https:}\webLink{#1}}

\title{\Large\bf{Stability for Inference with Persistent Homology Rank Functions}}

\author{Qiquan Wang\thanks{Department of Mathematics, Imperial College London, UK} \and In\'{e}s Garc\'{i}a-Redondo\thanks{London School of Geometry and Number Theory, University College London, UK}\textsuperscript{\, ,}\footnotemark[1] \and Pierre Faug\`{e}re\thanks{Department of Mathematics, ENS Lyon, France} \and Gregory Henselman-Petrusek\thanks{Pacific Northwest National Laboratory, Richland, Washington, USA} \and Anthea Monod\footnotemark[1]}
\date{}

\begin{document}
\maketitle

\begin{abstract}
 Persistent homology barcodes and diagrams are a cornerstone of topological data analysis that capture the ``shape'' of a wide range of complex data structures, such as point clouds, networks, and functions.  However, their use in statistical settings is challenging due to their complex geometric structure. In this paper, we revisit the persistent homology rank function, which is mathematically equivalent to a barcode and persistence diagram, as a tool for statistics and machine learning. Rank functions, being functions, enable the direct application of the statistical theory of functional data analysis (FDA)---a domain of statistics adapted for data in the form of functions. A key challenge they present over barcodes in practice, however, is their lack of stability---a property that is crucial to validate their use as a faithful representation of the data and therefore a viable summary statistic. In this paper, we fill this gap by deriving two stability results for persistent homology rank functions under a suitable metric for FDA integration. We then study the performance of rank functions in functional inferential statistics and machine learning on real data applications, in both single and multiparameter persistent homology. We find that the use of persistent homology captured by rank functions offers a clear improvement over existing non-persistence-based approaches.
\end{abstract}


\section{Introduction}
\label{sec:introduction}
Topological data analysis (TDA) leverages theory from algebraic topology to computational and data analytic settings, and has enjoyed great success in applications in many fields, including biology \citep{emmett_multiscale_2016, cang_analysis_2017, cang_representability_2018}, medicine \citep{crawford2016functional, biwer_windowed_2017}, physics \citep{de_silva_coverage_2007}, economics \citep{gidea_topological_2017}, and motion planning \citep{bhattacharya_persistent_2015, vasudevan_persistent_2013}, to name a few.  A cornerstone methodology of TDA is \emph{persistent homology}, which produces a summary statistic of data.   Its widespread applicability stems from its flexibility to adapt to a variety of complex data structures; its interpretability within the scientific domains where data arise; and its stability, which is the focus of this work.  Stability provides a notion of faithfulness of the topological representation of the input data, guaranteeing that a bounded perturbation of the input data results in a bounded perturbation of their topological representation captured by persistent homology.  Stability is a crucial property that validates the use of persistent homology in real data applications.

Persistent homology has its roots in various constructions and thus has various representations; the most well-known is the \emph{persistence diagram} or equivalently, the \emph{barcode}.  A less-used representation is the \emph{rank function}, initially proposed in the early 1990s \citep{frosini1992measuring} as the \emph{size function}.  Size functions were used as a mathematical tool for shape and image analysis in computer vision and pattern recognition \citep{Verri1993, frosini1999size, landi2002size, Biasotti2008, 10.1007/978-3-642-04146-4_69}, and were reinterpreted in algebraic terms using a correspondence between size functions and formal series \citep{frosini_size_2001} (see \cite{Biasotti:2008:DSG:1391729.1391731} for a thorough survey on the theory of size functions). Such algebraic topological formulations of rank functions directly coincide with parallel concepts from persistent homology.

Although rank functions were proposed before persistence diagrams and barcodes, and they are in fact equivalent to them, their use has been comparatively restricted due to a the difficulty in establishing stability results and a less comprehensive understanding of their effectiveness in practical data scenarios.  
Nevertheless, with the current active research interest in \emph{multiparameter persistent homology}, rank functions are again becoming increasingly relevant, as they inherently and directly adapt to this higher dimensional framework where persistence diagrams and barcodes do not \citep{carlsson_zomorodian_multiparameter2009}.  Additionally, unlike persistence diagrams and barcodes, rank functions, due to their structural form as functions, are naturally amenable to \emph{functional data analysis} (FDA), a robust statistical field focused on analyzing data taking the form of functions, curves, and surfaces. This adaptability to FDA (see \cite{ramsay2, ramsayapplied2} for a thorough introduction to the field and its applications) offers a rich statistical toolkit not directly accessible with persistence diagrams and barcodes.
FDA methods have been previously used in TDA by \cite{crawford2016functional}, who perform Gaussian process regression using a summary statistic constructed from a dynamic version of the Euler characteristic, which is an alternative topological invariant and distinct from persistence barcodes and diagrams. Principal component analysis (PCA) for rank functions---an important dimension reduction technique in descriptive, rather than inferential, statistics---has also been studied in \cite{robins_principal_2016}. 

In our work, we aim to study the performance of rank functions in inference tasks, which move beyond the descriptive analysis by \cite{robins_principal_2016} and, within the field of statistics, are arguably significantly more challenging: where descriptive statistics studies properties observed in a single sample of data, inferential statistics aims to impute information and provide guarantees on the possibly infinite, unobserved population from observed data.  To validate our findings, however, we need to first establish suitable stability properties of rank functions.  Here, ``suitable'' means that we will need to establish stability under a metric conducive to the application of FDA methods.  Once this is achieved, we then study the performance of rank functions in both single- and multiparameter persistent homology in inferential machine learning tasks on real data applications.  We find a clear improvement in performance using rank functions compared to existing methods that do not incorporate persistent homology, and other persistence-based methods.  

The remainder of this paper is organized as follows. In Section \ref{sec:setting}, we provide a background and literature review on rank functions and discuss their relation to barcodes and persistence diagrams, as well as various metrics associated to these representations and their implications on stability.  In Section \ref{sec:stability_results}, we present two stability results for rank functions with respect to a suitable metric for FDA implementation.  These stability guarantees motivate the applications of rank functions in inferential tasks using FDA to real data presented in Sections \ref{sec:application} and \ref{sec:mph_applications}. We close with a discussion in Section \ref{sec:conclusion} on our findings and propose directions for future research.  Proofs and further theoretical details are given in Appendix \ref{app:proofs}.

\section{Preliminaries: Persistent Homology}
\label{sec:setting}

In this section, we review the essential background and existing literature to the theory of persistent homology and its metrics, foundational to our work.

\subsection{Persistence Modules and Rank Functions}
\label{sec:persistence}

The algebraic object that is the central focus of persistent homology theory is the \emph{persistence module}, a functor mapping from a poset category to the category of vector spaces, \(M:(\P, \leq) \to \Vect\), also written as \(M \in \Vect^{(\P,\leq)}\) \citep{bubenik_categorification_2014, bubenik_metrics_2015, kim_generalized_2021}. 
Unless otherwise specified, we assume that \(\Vect\) is the category of finite dimensional vector spaces and work with \emph{pointwise finite dimensional} (p.f.d.) persistence modules. 

Arguably, the most relevant example is the module of persistent homology for a finite simplicial complex, first introduced by \cite{edelsbrunner_topological_2002} and obtained as follows. Consider a \emph{filtration}, i.e., a diagram \(F \in \mathsf{Simp}^{(\R,\leq)}\) such that \(F_x := F(x)\) is a finite simplicial complex for each \(x\in \R\), and such that for any other \(y\in \R\) with \(x\leq y\), \(F(x\leq y)\) is an inclusion \(F_x \subset F_y\). A common example is the Vietoris--Rips filtration \citep{vietoris_uber_1927}, which for a metric space $(X,d)$ and a finite subset \(S \subset X\) is denoted as $\mathrm{VR}(S) = \{\mathrm{VR}_t(S):\, t \in [0, +\infty)]\}$. The simplicial complex at filtration value \(t\in \R\) is defined as the family of all simplices of diameter less or equal than \(t\) that can be formed with the finite set \(S\) as set of vertices. An example of a Vietoris--Rips filtration is given in Figure \ref{fig:vietoris-rips}.

Given that we are working with finite simplicial complexes, there is a finite discrete set of values over which the filtration changes.
For each \(k\geq 0 \), we obtain another diagram \(\H_k(F) \in \mathsf{Vec}^{(\R,\leq)}\) by setting
\[\H_k(F)\, (x) := \H_k(F_x, \Fbb),\]
where \(\H_k: \mathsf{Simp} \to \mathsf{Vec}\) is the homology functor of order \(k\geq 0 \) with coefficients over a field \(\Fbb\).
We shall also refer to the diagram given by
\[\H(F)\, (x) := \H(F_x, \Fbb) = \bigoplus_{k\geq 0} \H_k(F_x,\Fbb)\]
using the notation \(\H(F)\in \mathsf{Vec}^{(\R,\leq)}\).

\begin{figure}[htp]
    \centering
    \includegraphics[width= .9\columnwidth]{figures/filtration_circle_N_30.png}
    \caption{Vietoris--Rips filtration of a point cloud of 30 points over a circle with Gaussian noise of scale \(0.1\) added, at 4 different filtration values (i.e.,~radius of the balls centered on the points). The simplicial complexes are built by adding an edge between any two points with overlapping circles (at a distance less or equal than twice the filtration value); and higher \(k\)-dimensional simplices for each \(k+1\) subset of connected points. From the third image to the fourth, a 1-cycle (loop) appears in the filtration, encircling the hole in the shape. This is the type of feature that PH aims to capture.}
    \label{fig:vietoris-rips}
\end{figure}

The p.f.d.~persistence module \(\H(F)\in \Vect^{(\R,\leq)}\) is called the \emph{persistent homology} module of the filtration. The previous construction is dimensionality-independent, allowing its application to both 2D and 3D data: the Vietoris--Rips filtration relies solely on pairwise distances, making it adaptable to any dimension. We demonstrate its use with 2D data in Section \ref{sec:application} and with 3D data in Section \ref{sec:mph_applications}.

Observe that we obtain a spectrum of linear maps connecting the vector spaces, and thus, a natural way to study its structure is to consider the \emph{ranks} of these maps. Let \(\R^{2+}:= \left\{(x,y) \in \left(\{-\infty\} \cup \R\right) \times \left(\R \cup \{\infty\}\right): x\leq y\right\}.\)

\begin{definition}[Rank Function]
\label{def:rank_function}
Given a p.f.d.~persistence module \(M\in \Vect^{(\R,\leq)}\), its \emph{rank function} is defined as 
\[\begin{array}{crcl}
    \beta^{M} : &  \R^{2+} & \to & \Zbb\\
     &  (x,y) & \mapsto & \rank M(x\leq y) = \dim \Im (M(x\leq y)).
\end{array}\]
The space of rank functions will be denoted by \(\mathcal{I}_1\).
\end{definition}

Deterministic and probabilistic properties of the rank functions have been also studied under the name of \emph{persistent Betti numbers}, first introduced by \cite{Edelsbrunner03}.  
For instance, \cite{duy2016limit, krebs2023asymptotic} investigated the asymptotic normality and stabilizing properties of central limit theorems of persistent Betti numbers under the homogeneous Poisson and binomial processes and variants. This allows for implementations of the bootstrap procedure on the persistent Betti numbers \citep{Roycraft_2023}. Further, in \cite{botnan2021consistency} the consistency and asymptotic normality of multiparameter persistent Betti numbers in large domains was determined, which is an important foundation to constructing statistical hypothesis tests. It is worth highlighting that this line of work deals mainly with \emph{probabilistic} properties of rank functions (i.e., persistent Betti numbers), as opposed to their \emph{statistical} performance in real data analysis, which is the focus of this work.

\subsection{Persistence Diagrams and Barcodes}
\label{sec:pd_and_barcodes}

A \emph{complete invariant} is a specific invariant assigned to a persistence module: it has the same value for all isomorphic persistence modules and is different for non-isomorphic ones. In single-parameter persistent homology, rank functions are equivalent to \emph{persistence diagrams} and \emph{barcodes}---two complete, discrete invariants obtained from the following distinct approaches.

Persistence diagrams can be traced back to the study of discontinuities of rank functions \citep{frosini_size_2001}, later reinterpreted to visually capture the persistent homology of a filtered simplicial complex \(H(F) \in \Vect^{(\R,\leq)}\) \citep{edelsbrunner_topological_2002}. We introduce them for general persistence modules \(M \in \Vect^{(\R,\leq)}\). 
Let \(T= \{t_1,\ldots,t_\ell\} \subset \R\) be the discrete set of values over which the module changes,
and consider a sequence \(\{s_0,s_1,\ldots,s_\ell\}\) of real numbers interleaved with the elements of $T$: \(s_{i-1}\leq t_i \leq s_i\). Also, set \(s_{-1} = t_0 = -\infty\) and \(s_{\ell+1} = t_{\ell+1} = +\infty\) and call \(\overline{T} = T \cup \{- \infty, + \infty\} \). Lastly, define
\begin{equation}
    \mu_i^j := \beta^M(s_{i-1},s_j) - \beta^M(s_i,s_j) + \beta^M(s_{i},s_{j-1}) - \beta^M(s_{i-1},s_{j-1}).
    \label{eq:inclusion_exclusion}
\end{equation}

\begin{definition}
The {\em persistence diagram} of $M \in \Vect^{(\R, \leq)}$ is the multiset of points given by
\[\Dgm (M) := \{(t_i, t_j)\in \overline{T} \times \overline{T} : t_i < t_j\},\]
where each point \((t_i, t_j)\) has multiplicity \(\mu_i^j\), union all the points in the diagonal \(\partial = \{(x,y) \in \R^{2+}: x = y\}\) counted with infinite multiplicity. The space of persistence diagrams is denoted by \(\mathcal{D}\).
\label{def:persistence_diagram}
\end{definition}

Figure \ref{fig:3d_persistence_diagrams} shows persistence diagrams for 0- and 1-homology, representing components and loops, respectively, in three 3D point clouds. The sphere exhibits no persistent points far from the diagonal, indicating the absence of non-trivial loops on its surface---every loop can be deformed to a point. In contrast, the torus displays two persistent points representing its horizontal and vertical loops. The Stanford Bunny \citep{turk_zippered_1994} (see \httpsAddr{//faculty.cc.gatech.edu/~turk/bunny/bunny.html}) exhibits a more complex interpretation, featuring one persistent 1-cycle.

The Structure Theorem due to  \cite{zomorodian_computing_2005} (with the version in full generality due to \cite{crawley-boevey_decomposition_2015} and \cite{botnan_decomposition_2020}) asserts that any p.f.d.~persistence module $M$ is isomorphic to an essentially unique (up to reordering) finite direct sum of indecomposable persistence modules
$$
M \simeq M_1 \oplus \cdots \oplus M_m.
$$
For single-parameter persistence, each $M_j$ is an {\em interval persistence module}, i.e., there is a pair of values $b_j<d_j$, where $d_j$ may be infinite, such that $M_j(x)$ is a copy of the field for all values $b_j \leq x< d_j$ and zero elsewhere. We denote these persistence modules by \(\mathbb{I} [b_j,d_j) \). 

\begin{definition}[\hspace{1sp}\citep{zomorodian_computing_2005, carlsson2010zigzag}]
\label{def:barcode}
The \emph{barcode} of $M \in \Vect^{(\R, \leq)}$ is the list of indecomposables given by the Structure Theorem, or equivalently, the collection of intervals that define these indecomposables
\[\Barc (M) := \{[b_j, d_j):\, M_j(x) \mathrm{\ is \ non trivial\ for\ } b_j \leq x < d_j\}.\]
The length of a bar is called its \emph{persistence}.
\end{definition}

A barcode translates to a persistence diagram by plotting the left and right endpoint of each interval persistence module as an ordered pair.  A persistence diagram translates to a barcode by turning each point $(x,y)$ with \(x<y\) in the persistence diagram into an interval persistence module beginning at $x$ and ending at $y$.  In this way, the persistence diagram is equivalent to a barcode, although the two definitions arise from different perspectives. 

\begin{figure}[H]
    \centering
    \begin{subfigure}{0.45\linewidth}
        \centering
        \includegraphics[width=.9\linewidth]{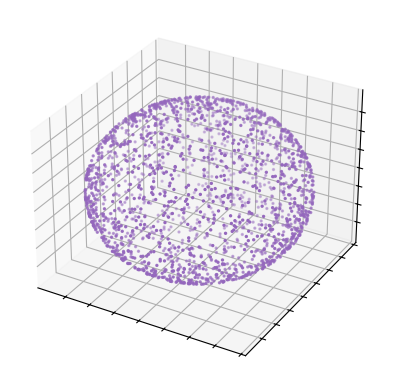}
        \caption{2000 points on a 3D sphere.}
    \end{subfigure}
    \hfill
    \begin{subfigure}{0.45\linewidth}
        \centering
        \includegraphics[width=.9\linewidth]{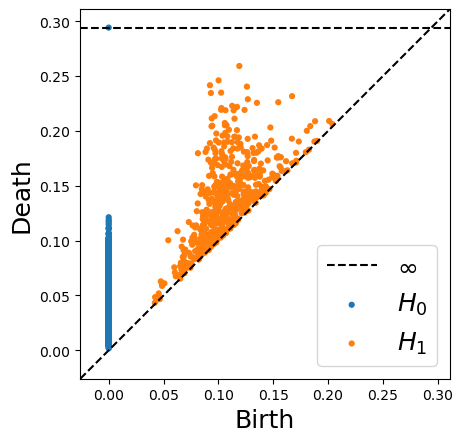}
        \caption{Corresponding diagram.}
    \end{subfigure}
    \newline
    \begin{subfigure}{0.45\linewidth}
        \centering
        \includegraphics[width=.9\linewidth]{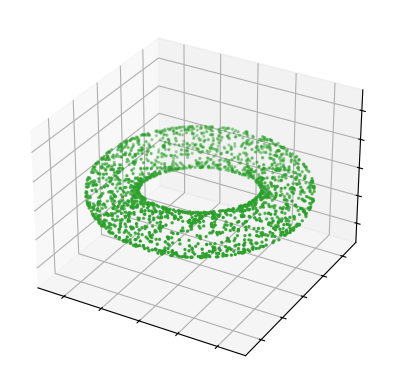}
        \caption{2000 points on a 3D torus.}
    \end{subfigure}
    \hfill
    \begin{subfigure}{0.45\linewidth}
        \centering
        \includegraphics[width=.9\linewidth]{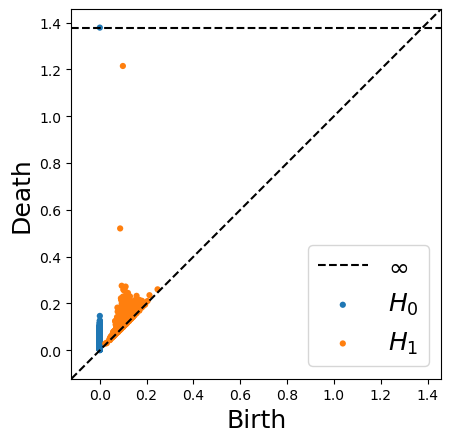}
        \caption{Corresponding diagram.}
        \label{fig:plot4}
    \end{subfigure}
    \newline
    \begin{subfigure}{0.45\linewidth}
        \centering
        \includegraphics[width=.9\linewidth]{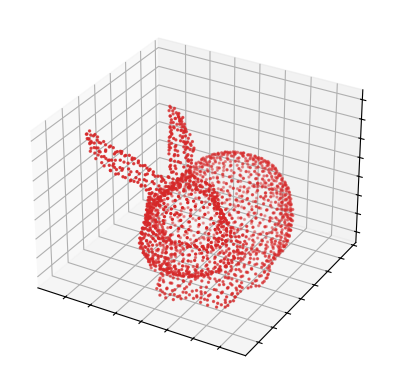}
        \caption{Stanford bunny with 1889 points.}
    \end{subfigure}
    \hfill
    \begin{subfigure}{0.45\linewidth}
        \centering
        \includegraphics[width=.9\linewidth]{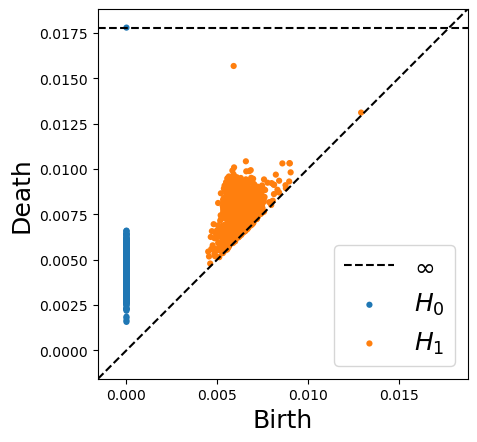}
        \caption{Corresponding diagram.}
    \end{subfigure}
    \caption{Examples of 3D point clouds and their corresponding persistence diagrams.}
    \label{fig:3d_persistence_diagrams}
\end{figure}

Rank functions are also in bijection with persistence diagrams and barcodes. We have seen above how to define persistence diagrams from an inclusion-exclusion formula \eqref{eq:inclusion_exclusion} on rank functions. Moreover, rank functions can be seen as cumulative functions on the persistence diagrams: the value of the rank function \(\beta^{M}(x,y)\) corresponds to the number of points (counted with multiplicities) in the region \((-\infty,x] \times [y, \infty)\) of the persistence diagram, providing the converse direction of the bijection. Figure \ref{fig:pdtorank} illustrates a persistence diagram and its corresponding rank function, where the equivalence between both objects becomes apparent.

\begin{figure}[htb]
\begin{subfigure}{.5\columnwidth}
\centering
\subfloat{\includegraphics[width=.9\columnwidth]{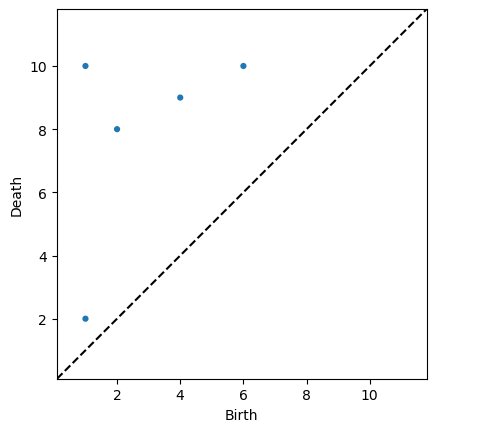}}
\end{subfigure}%
\begin{subfigure}{.5\columnwidth}
\centering
\subfloat{\includegraphics[width=.9\columnwidth]{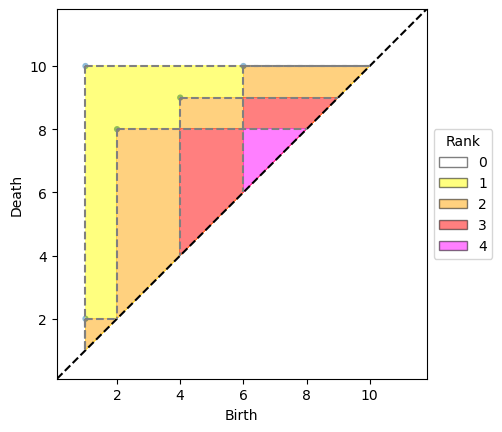}}
\end{subfigure}
\caption{Persistence diagram and its corresponding rank function.}
\label{fig:pdtorank}
\end{figure}

\subsection{Metrics and Stability in Persistent Homology}
\label{sec:stability}

Stability results, which are the main theoretical contribution of this paper, give bounds for metrics defined on invariants of persistence modules.  Given that there exist various metrics in persistent homology, the question of choosing appropriate metrics depends on the eventual goal.

\noindent
\textbf{Metrics on Barcodes and Persistence Diagrams.} We recall the most widely used metrics to compare persistence diagrams in single-parameter persistent homology and their stability properties.

\begin{definition}
The \emph{bottleneck distance} between the persistence diagrams \(D_1\) and \(D_2\) is defined as
\[d_B(D_1, D_2) := \inf_{\phi: D_1 \to D_2}\sup_{x \in D_1}\norm{x- \phi(x)}_\infty,\]
where \(\|\cdot\|_\infty\) is the infinity norm \(\ell^\infty\) on \(\R^2\) and \(\phi\) ranges over all bijections between \(D_1\) and \(D_2\).
\end{definition}

The first stability result in TDA concerns the bottleneck distance \citep{cohen-steiner_stability_2007}, proving that the map sending a tame function \(f:X\to \R\) defining a filtration \(F(x) := f^{-1}(-\infty,x]\) to its persistence diagram \(\Dgm \H(F)\) is \(1\)-Lipschitz with respect to the bottleneck distance for diagrams and the \(L^\infty\) distance for functions. 

The bottleneck distance can be extended by replacing \(\ell^\infty\) with \(\ell^p\) norms, which give a stronger sense of proximity.
\begin{definition}
\label{def:pwass_dist}
For \(1 \leq p < \infty\) and \(1\leq q \leq \infty\) we define the \emph{p,q-Wasserstein distance} between two persistence diagrams \(D_1\) and \(D_2\) as
\[W_{p,q}(D_1, D_2) := \inf_{\phi: D_1 \to D_2} \left[\sum_{x \in D_1}\norm{x - \phi(x)}_q^p\right]^{1/p}\]
where \(\| \cdot \|_q\) is the \(\ell^q\) norm on \(\R^2\) and \(\phi\) ranges over all bijections between \(D_1\) and \(D_2\).\
\end{definition}
In our work, we will assume \(p = q\) and refer to this metric simply as the \(p\)-Wasserstein distance \(W_p\).

Wasserstein metrics are widely used in applications, especially \(p\)-Wasserstein distances with \(p=1,2\), and in some instances have been able to reveal more insight in application settings than the bottleneck distance.  For example, in a protein flexibility analysis study, \cite{bramer_atom-specific_2020} show that \(p,\infty\)-Wasserstein metrics provide more accurate results when comparing two conjugate point clouds obtained using atom-specific persistent homology \citep{cang_analysis_2017, cang_representability_2018}. \cite{gamble_exploring_2010} explore the power of the Wasserstein distance in the context of statistical analysis of landmark-shape data. 
 \cite{gidea_topological_2017} shows that the \(W_{2,\infty}\) distance is able to detect premature evidence for critical transitions is financial data.  \cite{hamilton_persistent_2022} study barcodes endowed with Wasserstein distances for protein folding data and compare them with the Gaussian integral tuned \citep{burley_rcsb_2021} vector representation.

Despite the desirable performance of Wasserstein distances in applications, their stability properties have not been as broadly studied 
until recently, due to the level of additional technicality required. In \cite{skraba_wasserstein_2021}, the following cellular Wasserstein Stability Theorem for \(p\)-Wasserstein metrics was established, validating many of the existing results in applications and justifying their continued use.

\begin{theorem}[Cellular Wasserstein Stability Theorem, \citep{skraba_wasserstein_2021}]
    \label{thm:cellular_wass_stability}
    Let \(f,g: K \to \R\) be monotone functions on a finite CW-complex \(K\). Then
$
        W_p(\Dgm(f) ,\, \Dgm(g)) \leq \| f-g\|_p.
$
\end{theorem}
The existence of this result also justifies comparison with the Wasserstein metric to establish stability. 

\noindent
\textbf{Metrics on Rank Functions.} Rank functions lie in the space of functions from \(\R^{2+}\) to \(\R\) \citep{robins_principal_2016}. By fixing a metric on the space \(\R^{2+}\), we can then define the \(L^p\) metric on this space of functions as follows:
\[d_{L^p}(f,g) := \left(\int_{\R^{2+}} (f-g)^p d\omega\right)^{1/p},\]
where \(\omega\) is the measure on \(\R^{2+}\) corresponding to the fixed metric. For notational simplicity, we write 
\(\norm{f-g}_p = d_{L^p} (f,g),\)
keeping in mind that this metric comes from the \(L^p\) norm. This metric space is naturally endowed with a Hilbert structure for \(p=2\), which is a basic requirement for many FDA methodologies.

There are many choices for the metric as well as for the measure \(\omega\), however, the choice should avoid the pairwise distance between rank invariants being infinite. This happens, for example, when the metric is taken to be the Euclidean distance restricted to \(\R^{2+}\), which implies \(\omega\) is the Lebesgue measure.  Here, two rank invariants have finite distance if and only if their infinite cycles have identical birth times.

\begin{remark}
\label{rmk:infinite_distance_solved}
In our work, such issues are circumvented by keeping in mind the goal of real data applications, where the posets over which we are defining our diagrams are always finite and in which the filtrations always end in a simplicial complex with trivial homology. This means that every cycle in the filtration is destroyed at some point, except for the \(0\)-cycle representing the connected component of the space, which is always born at time \(0\). Thus, we can work with the Lebesgue measure without worrying about infinite distances between rank invariants: all bars in our barcodes will be finite, and therefore, every two rank functions will have finite pairwise distance, as desired. 
\end{remark}


\section{$L^p$-Stability of Rank Functions}
\label{sec:stability_results}

In this section, we present our contributions of two stability guarantees for rank functions endowed with $L^p$ metrics with respect to the bottleneck distance and 1-Wasserstein distance for persistence diagrams.  We focus on the $L^p$ metric, as opposed to \cite{skraba_wasserstein_2021}, who study the weighted version,
\begin{equation}
\label{eq:lp_weight}
\norm{f - g}_p^{\text{w}} = \left(\int_{\R^{2+}} \abs{f - g}^p \phi(x - y)\, dx\, dy\right)^{1/p},
\end{equation}
where the weight function $\phi(\cdot)$ inside the integral satisfies \(\int_{\R} \phi(t) < \infty\). In particular, they choose \(\phi(t) = e^{-t}\) and obtain the following stability result as a Corollary from Theorem \ref{thm:cellular_wass_stability}.

\begin{corollary}[\hspace{1sp}\citep{skraba_wasserstein_2021}]
    Rank functions with the \(L^q\) weighted metric \eqref{eq:lp_weight} are 1-Lipschitz with respect to the \(p\)-Wasserstein distance between diagrams if and only if \(p=q=1\).
    \label{cor:L^p_weighted_stability}
\end{corollary}

The weight function in \eqref{eq:lp_weight} ensures finite distances between rank functions \cite[Lemma 2.1.]{robins_principal_2016}, which allows for the definition of an inner product structure on finite sets of rank functions and thus justifies the rank FPCA method proposed in \cite{robins_principal_2016}. In our study, finiteness of $L^p$ distances between rank functions is guaranteed  as explained in Remark \ref{rmk:infinite_distance_solved}, and the Hilbert space structure follows directly. 
As a result, the use of such a weight is not needed in our setting, and in fact it is not helpful for our purposes: its introduction fundamentally changes the expressions in the computations involving the metric, so that the proof of Corollary \ref{cor:L^p_weighted_stability} by \cite{skraba_wasserstein_2021} does not apply and cannot be replicated in our work.  This also underlines the inherent differences between our contribution compared to \cite{skraba_wasserstein_2021}.

\noindent
\textbf{Other Hindrances to Rank Function Stability.}
Under the original name of size functions \citep{frosini1992measuring}, several notions of ``pseudo-stability'' were established. These were achieved under \emph{pseudometrics}, where the distance between two distinct points can be zero, which makes pseudometrics less desirable for use in real data analyses (due to difficulties in intepretability) and therefore also makes pseudo-stability less desirable as a validating property to justify the use of rank functions as topological summary statistics.  Some examples are the \emph{deformation distance}, which were adapted to persistence diagrams and is more widely known today as the 1-Wasserstein distance between persistence diagrams; the \emph{Hausdorff pseudo-distance}, which similarly gave rise to the bottleneck distance between persistence diagrams; and the $L^p$ \emph{pseudo-distance}, which exhibited an unstable nature and did not appear to inspire any well-known distance in persistent homology. 

In particular, it is worth noting that \cite{damico_2003_optimal} and \cite{damico_2010_natural} renamed the Hausdorff pseudo-distance to the \emph{matching distance} to emphasize the fact that its computation amounts to finding an optimal matching between multisets, in the same way that the bottleneck distance does, and which was used in establishing the first results of stability for persistent homology. The matching distance was used to establish stability under noisy perturbations when restricted to a subset of size functions called \emph{reduced} size functions.

\subsection{Stability Under the Bottleneck Distance}

The most straightforward way to achieve stability for rank functions is to restrict away from the diagonal, which is known to complicate the metric geometry of the space of persistence diagrams \citep{Turner2014,cao2022geometric}.  To do this, we introduce a truncation of the rank function that will allow us to compare its sensitivity to noise to that of the bottleneck distance.

\begin{definition}
For any rank function $\beta$ and any $\delta >0$, we define the \emph{$\delta$-truncated rank function} as 
\[
\beta_{\delta}:=\beta \cdot \mathds{1}_{\R^{2+}_\delta}
\]
where \(\mathds{1}_{\R^{2+}_\delta}\) is the indicator function of the set \(\R^{2+}_\delta : =
{\{(x,y) \in \mathbb{R}^{2+} \, : \, y>x+\delta \}}\). 
\label{def:trucated_rank_function}
\end{definition}

In other words, the truncated rank function is just the rank function excluding a strip of width \(\delta>0\) above the diagonal \(\partial\) (see Definition \ref{def:persistence_diagram}). The truncated rank function locally satisfies a H\"{o}lder inequality for the $L^p$ norm with respect to the bottleneck distance on persistence diagrams. 

\begin{proposition}[Bottleneck Stability for Truncated Rank Functions]
Let \(1 \leq p < \infty\) and $M$ be a p.f.d. persistence module with finite intervals in its barcode decomposition. For every $\delta>0$, there exist $1\geq\eta>0$ and $K_{M,p}>0$ such that any persistence module $N$ satisfying
\[
d_B(\Dgm(M),\,\Dgm(N)) < \eta
\]
also satisfies
\begin{equation}
    \norm{\beta_\delta^M-\beta_\delta^N}_{p} \leq K_{M,p} \cdot d_B(\Dgm(M),\, \Dgm(N))^{1/p}.
\label{eq:estimate_truncanted_rank_function}
\end{equation}
In other words, the map \((\mathcal{D}, d_B) \to (\mathcal{I}_1, L^p)\) which sends each persistence diagram to its corresponding rank function is locally H\"{o}lder with exponent \(1/p\).
\label{prop:stability_truncated}
\end{proposition}

\begin{remark}
    The constant appearing in Proposition \ref{prop:stability_truncated} is precisely
    \begin{equation}
        K_{M,p} = m \, (2R + 2)^{1/p}
    \end{equation}
    where \(m\) is the number of points in \(\Dgm(M)\) and \(R= \max\{|d_i - b_i|: 1 \leq i \leq m\}\) is the maximum persistence in such diagram.
\end{remark}

Notice that we can always obtain a bound similar to that in \eqref{eq:estimate_truncanted_rank_function} where the constant depends on both persistence modules \(M\) and \(N\) (see Appendix \ref{app:proofs}). Proposition \ref{prop:stability_truncated} refines this approach by obtaining a constant that only depends on the persistence module \(M\). 
Nevertheless, an important limitation of Proposition \ref{prop:stability_truncated} is that it discards the points close to the diagonal---an important component in the definition of persistence diagrams---even though it sheds light on the behavior of rank functions in discrete settings. 
The bounds appearing in the proof (see Appendix) will be useful in our next derivations.

\subsection{Stability Under the 1-Wasserstein Distance}

The previously-mentioned limitation of Proposition \ref{prop:stability_truncated} is that it holds only for points away from the diagonal, which highlights the differences in sensitivity to noise between the $L^p$ norms on rank functions and the bottleneck distance on persistence diagrams. This observation was already made by \cite{landi1997new}; we further develop this observation with a formal study in this paper. 

As we will now establish, \emph{full} rank functions satisfy a stability property with respect to the 1-Wasserstein distance for persistence diagrams. As mentioned in Section \ref{sec:stability}, stability properties of the Wasserstein metric on persistence diagrams were not studied in detail until very recently, which limited their applicability as upper bounds in stability studies. We use the Cellular Wasserstein Stability Theorem (Theorem \ref{thm:cellular_wass_stability}  \cite{skraba_wasserstein_2021}) to establish stability for rank functions.

\begin{theorem}[1-Wasserstein Stability for Rank Functions]
Let \(p = 1,\, 2;\) and $M$ be a p.f.d.~persistence module with finite intervals in its barcode decomposition. Then there exists a constant $C_{M,p}>0$ such that for any other p.f.d.~persistence module $N$ satisfying
\(W_1(\Dgm(M),\, \Dgm(N)) \leq 1,\)
we have
\begin{equation}
    \norm{\beta^M-\beta^N}_{p} \leq C_{M,p} \cdot W_1(\Dgm(M),\, \Dgm(N))^{1/p}.
    \label{eq:estimate_wasserstein}
\end{equation}
In other words, the map $(\mathcal{D},\, W_1) \rightarrow (\mathcal{I}_1,\, L^1)$ sending a persistence diagram to its corresponding rank function is locally Lipschitz, and the same map between the spaces $(\mathcal{D},\, W_1) \rightarrow (\mathcal{I}_1,\, L^2)$ is locally H\"{o}lder with exponent \(1/2\).
\label{thm:stability_wasserstein}
\end{theorem}

\begin{remark}
    The constants appearing in Theorem \ref{thm:stability_wasserstein} are
    \begin{equation}
        C_{M,1} = 2R + 2,
    \end{equation}
    and
    \begin{equation}
       C_{M,2} =2 \cdot \max\left\{\left(2\, (R+1)\, m\right)^{1/2},\, \frac{1}{\sqrt{2}}\right\} 
    \end{equation}
    where \(m\) is the number of points in \(\Dgm(M)\) and \(R= \max\{|d_i - b_i|: 1 \leq i \leq m\}\) is the maximum persistence in such diagram.
\end{remark}

Theorem \ref{thm:stability_wasserstein} provides a stronger theoretical guarantee than Proposition \ref{prop:stability_truncated}, not only because it also covers the diagonal, but also because the constant \(C_{M,p}\) in \eqref{eq:estimate_wasserstein} is smaller than the constant \(K_{M,p}\) in \eqref{eq:estimate_truncanted_rank_function} (\(p =1,2)\). For \(p=1\), the latter depends on the number of points in the persistence diagram of \(M\) and \(R\), its maximum persistence, whereas the former only depends on \(R\). For \(C_{M,2}\) we maintain a dependence on the number of points in the diagram of \(M\), but it still provides a tighter bound than \(K_{M,2}\), since this dependence is squared instead of linear.


\section{Inference with Rank Functions}
\label{sec:application}

In this section, we study the performance of rank functions in machine learning tasks on real data.  Specifically, we focus on \emph{inferential} tasks---namely, classification and prediction---in the single-parameter setting.  

\subsection{Using Persistent Homology in Data Analysis}

Persistent homology captured by persistence diagrams and barcodes fulfils the essential requirements for data analysis: interpretability (via the Structure Theorem) and stability (with various stability results available).  Moreover, it is known to be a viable space for probability and statistics \citep{0266-5611-27-12-124007, Blumberg2014}. Despite these desirable properties, there remain challenges in utilizing persistence diagrams and barcodes in the full scope of statistical analysis, mainly due to their complicated geometry which results in, for example, non-unique geodesics and Fr\'echet means \citep{Turner2014}. As a consequence, in statistical questions there are, broadly speaking, two approaches to handling persistent homology in data analysis.  One approach entails developing new data analytic methodology, such as machine learning algorithms and statistical models, to accommodate barcodes or persistence diagrams directly (e.g., \cite{fasy2014, reininghaus2015stable, NIPS2017_6761}).  The other approach entails vectorizing barcodes persistence diagrams to apply existing methods (e.g., \cite{chazal2014stochastic, bubenik2015statistical, adams2017persistence}). 

Our approach diverges from both strategies by exploring rank functions as equivalent, alternative representations of persistent homology to leverage theory from functional data analysis (FDA). Methods for FDA, which were constructed to analyze data in the form of functions, are well-established in statistics and many of them arise as extensions of methodologies from multivariate data analysis.
Rank functions equipped with the $L^2$ metric, as discussed above, form a metric space 
that admits a Hilbert structure, and thus become a viable data structure amenable to integrating FDA with persistent homology. 

We emphasize here that we are not modifying the output of persistent homology, as vectorization methods do to persistence diagrams or barcodes, since rank functions are equivalent to barcodes and persistence diagrams.  Equally, we are not developing new methodology to accommodate persistent homology, as the existing field of FDA is directly applicable to persistent homology captured by rank functions.

\subsection{Functional Support Vector Machine on Single-Parameter Rank Functions}
\label{sec:FSVM}

Our first study is the performance of rank functions in the single-parameter persistence setting in classification.  Specifically, we study the clinical application of discerning heart rate variability between healthy individuals and post-stroke (acute ischemic) patients using a \emph{functional support vector machine} (FSVM) \citep{Rossi_2006}. 

\noindent
\textbf{Functional Data Analysis and High Dimensionality. } Functional data, where datasets are collections of functions, are inherently infinite dimensional, which means that the discrete realizations of the underlying surfaces or curves are high dimensional and can cause various problems, such as overfitting.

However, FDA methodologies are generally insensitive to dimensionality; they circumvent the problem of high dimensionality in broadly two ways. One way is via dimensionality reduction, where projecting the data onto a smaller collection of orthogonal bases produces lower dimensional vectors that are robust to discretizations. The choice of basis function depends on the underlying functions; e.g., the Fourier basis functions can be used to approximate functions that exhibit cyclical properties, while wavelet basis functions can be used to approximate functions that exhibit fluctuations. Alternatively, a data driven approach may be adopted, with one of the most commonly used techniques being \emph{functional} PCA (FPCA) \citep{dauxois1982asymptotic}, which is a descriptive technique (rather than inferential, which is the focus of our work), and has previously been implemented on rank functions by \cite{robins_principal_2016}.  FPCA works on the principle of finding orthonormal basis functions and projecting onto a finite subset of them with the greatest variation. 

Another alternative is to ensure within the construction of the methodology that it is invariant to the choice of grid points for evaluation, such that as the number of grid points increases, convergence to an appropriate result is guaranteed by construction. This is known as the \emph{refinement invariance principle} \citep{Cox2008PointwiseTW} and will not be our focus here.

\noindent
\textbf{Functional Support Vector Machines.  }
Classical SVMs are supervised binary classification methods that seek to find the optimal boundary in the feature space which distinguishes between observations of the two categories in a way such that the distance to the boundary from any data point is maximized. For our FSVM application on rank functions, let $\{f_1,f_2,\dots,f_N\}$ be a collection of centered, discretized rank functions with corresponding labels $(y_i)_{i=1,\dots,N} \in\{-1,1\}$ identifying the two groups. 
We adopt the \emph{soft margin} approach to determine the boundary for conventional reasons, as it is used in most computational packages. Here, ``soft'' refers to certain deviations from the boundary being allowed for classification in the approach. The boundary can be defined for some function $\psi \in \mathcal{H}$, $\mathcal{H}$ being a Hilbert space, and scalar $b\in \mathbb{R}$ as
$\langle \psi, f_i\rangle +b =0,$ such that 
$y_i(\langle \psi, f_i\rangle +b)\geq 1-\zeta_i, \, \forall\ i=1, \dots, n,$
where $\zeta_i$s are slack variables in the soft margin approach providing trade-offs between accuracy and overfitting. The optimal boundary is one which maximizes the margin given by $\frac{2}{\|\psi\|}$ and the optimization problem can be solved more easily through its dual formulation, by sequential minimal optimization \citep{platt}. 

Practically, however, not all data are linearly separable, in which case, a workaround is to project the data to higher dimensions where a clearer division between the two classes then becomes observable.  This technique is referred to as the \emph{kernel trick} \citep{boser}. Let $\phi$ be the projection.  Then the inner product from the previous optimization problem is replaced by a kernel function, i.e., $\langle \phi(f),\, \phi(g) \rangle$. Examples include the \emph{polynomial} kernel (of order $d$), $K(f,g)=(\langle f,g \rangle+1)^d$, and the \emph{Gaussian radial basis function} (GRBF) kernel, $K(f,g)=\exp (-\gamma \|f-g\|^2)$ \citep{Rossi_2006}. Their usage can be seen in a wide range of biomedical applications, for example, in the identification of PTSD patients based on resting state functional magnetic resonance imaging (fMRI) \citep{saba_2022} and also in the classification of brain functions on electroencephalographic signals \citep{xie_2008}. In these two examples, the GRBF kernel outperforms other kernels and classifiers. 

\noindent
\textbf{Data Description. } 
The dataset we study consists of 86 sequences of 512 beat-to-beat time intervals (RR series) extracted from electrocardiograms from a clinical study between two groups of people in a similar age category: one group of 46 healthy individuals used as control and one group of 40 patients who have recently experienced stroke episodes \citep{dgasecki2021,knarkiewicz2021}. Stroke patients generally show reduced heart rate variability compared to healthy individuals \citep{Lees2018HeartRV}. We aim to discern differences in heart rate variability between the two groups using persistent homology rank functions. For computations, we first linearly interpolated between the points in the RR series to construct continuous functions over time and then we constructed a sublevel set filtration based on the height function in the positive $y$-direction. We compute the zero-dimensional persistent homology rank function for each individual's RR series in the dataset. An example of steps in this process is visualized in Figure \ref{fig:rrrank}.

\begin{figure*}[htb]
\centering
\begin{subfigure}{.33\textwidth}
\centering
\subfloat[]{\label{rank:a}\includegraphics[width=\textwidth]{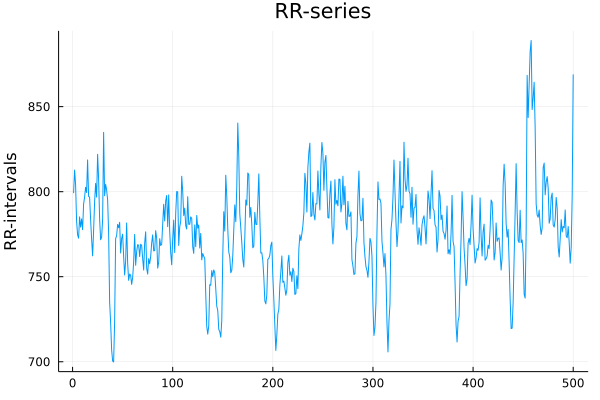}}
\end{subfigure}%
\begin{subfigure}{.33\textwidth}
\centering
\subfloat[]{\label{rank:b}\includegraphics[width=\textwidth]{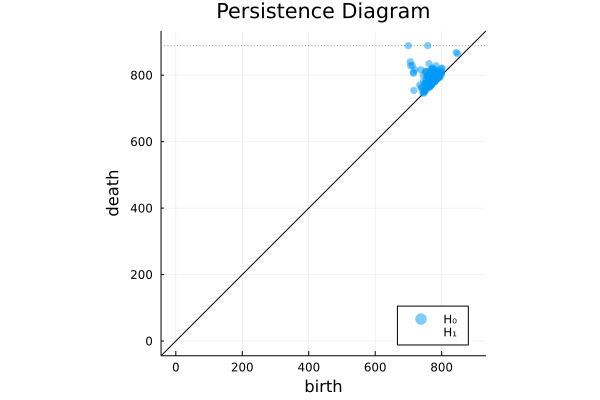}}
\end{subfigure}
\begin{subfigure}{.33\textwidth}
\centering
\subfloat[]{\label{rank:b}\includegraphics[width=\textwidth]{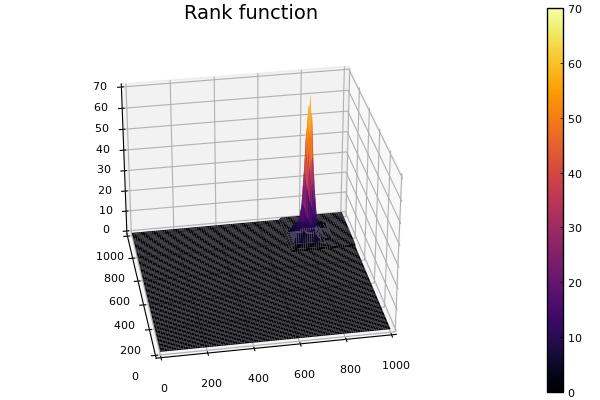}}
\end{subfigure}
\caption{An example showing (a) the RR series for a healthy individual, (b) its respective persistence diagram and (c) its rank function.}
\label{fig:rrrank}
\end{figure*}

\noindent
\textbf{Training the FSVM Classifier and Evaluating Performance. } On the set of rank functions computed from the data, we train FSVM classifiers using the linear kernel, GRBF kernel, and polynomial kernels of three different degrees (2, 3, and 5). Since we are working with discretized functions, we consider both the rank functions as computed from the data, and transformed versions using dimension reduction. We experiment with both a set of data-driven basis functions obtained from FPCA and a set of standard basis functions---the Haar wavelets.  Haar wavelets, in particular, have also been used in other inferential tasks in persistent homology; \cite{haberle2023wavelet} use them in persistent homology density estimation.

We evaluate the performance of the binary classifiers using two metrics: the accuracy and the area under the curve of the Receiver Operator Curve (AUC--ROC). The evaluation is carried out and averaged over ten iterations of five-fold cross-validation.

\noindent
\textbf{Results.}
\begin{table}[htb]
\centering
\caption{Average accuracy, AUC--ROC and runtimes  of classifiers constructed on rank functions and projected rank functions with linear, GRBF, and polynomial kernels over ten iterations of five-fold cross-validation.}
\label{table:results}
\begin{tabular}{c|ccc|}
\hline
\multicolumn{1}{|c|}{\multirow{2}{*}{Kernel}} & \multicolumn{3}{c|}{Discretized rank   functions + SVM} \\ \cline{2-4} 
\multicolumn{1}{|c|}{}                 & \multicolumn{1}{c|}{Accuracy} & \multicolumn{1}{c|}{AUC-ROC} & Runtime (s) \\ \hline
\multicolumn{1}{|c|}{Linear}           & \multicolumn{1}{c|}{39.5}     & \multicolumn{1}{c|}{0.408}   & 10.5        \\ \hline
\multicolumn{1}{|c|}{GRBF}             & \multicolumn{1}{c|}{75.8}     & \multicolumn{1}{c|}{0.762}   & 10.2        \\ \hline
\multicolumn{1}{|c|}{Polynomial (d=2)} & \multicolumn{1}{c|}{82.6}     & \multicolumn{1}{c|}{0.829}   & 8.51        \\ \hline
\multicolumn{1}{|c|}{Polynomial (d=3)} & \multicolumn{1}{c|}{81.6}     & \multicolumn{1}{c|}{0.829}   & 9.06        \\ \hline
\multicolumn{1}{|c|}{Polynomial (d=5)} & \multicolumn{1}{c|}{76.0}     & \multicolumn{1}{c|}{0.775}   & 8.28        \\ \hline
                                       & \multicolumn{3}{c|}{Projection on PCA basis + SVM}                         \\ \hline
\multicolumn{1}{|c|}{Linear}           & \multicolumn{1}{c|}{66.8}     & \multicolumn{1}{c|}{0.681}   & 53.2        \\ \hline
\multicolumn{1}{|c|}{GRBF}             & \multicolumn{1}{c|}{50.3}     & \multicolumn{1}{c|}{0.502}   & 3.14        \\ \hline
\multicolumn{1}{|c|}{Polynomial (d=2)} & \multicolumn{1}{c|}{84.2}     & \multicolumn{1}{c|}{0.842}   & 3.17        \\ \hline
\multicolumn{1}{|c|}{Polynomial (d=3)} & \multicolumn{1}{c|}{82.0}     & \multicolumn{1}{c|}{0.829}   & 3.14        \\ \hline
\multicolumn{1}{|c|}{Polynomial (d=5)} & \multicolumn{1}{c|}{75.5}     & \multicolumn{1}{c|}{0.774}   & 3.17        \\ \hline
                                       & \multicolumn{3}{c|}{Projection on Wavelet basis + SVM}                     \\ \hline
\multicolumn{1}{|c|}{Linear}           & \multicolumn{1}{c|}{39.7}     & \multicolumn{1}{c|}{0.404}   & 11.5        \\ \hline
\multicolumn{1}{|c|}{GRBF}             & \multicolumn{1}{c|}{75.0}     & \multicolumn{1}{c|}{0.756}   & 12.0        \\ \hline
\multicolumn{1}{|c|}{Polynomial (d=2)} & \multicolumn{1}{c|}{84.0}     & \multicolumn{1}{c|}{0.842}   & 9.99        \\ \hline
\multicolumn{1}{|c|}{Polynomial (d=3)} & \multicolumn{1}{c|}{80.3}     & \multicolumn{1}{c|}{0.816}   & 10.4        \\ \hline
\multicolumn{1}{|c|}{Polynomial (d=5)} & \multicolumn{1}{c|}{76.6}     & \multicolumn{1}{c|}{0.786}   & 10.0        \\ \hline
\end{tabular}%
\end{table}
The performance results of the FSVM classifier are summarized in Table \ref{table:results}. Overall, the FSVM classifiers with degree two and three polynomial kernels produce highest accuracy, $>80\%$, compared to the performance of other kernels implemented and gave on average AUC--ROC values of over 0.8, indicating excellent discrimination between the two categories. We also include the runtimes for these computations, including: (a) the computation of the PH rank functions; (b) the training of the corresponding SVM; and (c) the computation of 10 iterations of the accuracy and AUC--ROC over five-fold cross-validation, and the corresponding average. Although we observed a higher runtime for the linear SVM on PCA-projected data as well as convergence limitations within the specified maximum number of iterations, the computations are nevertheless manageable in terms of runtime.

In general, linear kernels perform less well in discriminating between functions of the two categories, except when the dimensionality of the rank functions is reduced by projecting onto its principal component functions.  In doing so, we considered only the first 30 principal component functions with the largest eigenvalues, which explains 95\% of the variation.  For the two transformations, working with lower dimensional vectors and polynomial kernels, we see similar, and slight improvements, in accuracy and AUC--ROC of the classifiers than the original rank function. 

Following \cite{graff2021persistent}, we take the SVM classifier quadruples of clinical indices related to the RR series and quadruples of features derived from persistence diagrams as our input data; we now discuss this analysis in further detail. The AUC--ROCs of the optimal models for both approaches can be found in Table \ref{table:comparison_rank_fns} as reported in \cite{graff2021persistent}. The optimized performance of the classifier using rank functions was, on average, better than the performance of standard heart rate variability indices on the frequency and time domain, which only achieved an average AUC--ROC of 0.79 and 0.75 respectively \citep{graff2021persistent}. For the persistence-based approach in \cite{graff2021persistent}, a wide range of topological indices were extracted from the persistent diagrams. Some indices were more typical, such as the total number of intervals, the sum of the lengths of all the persistent intervals, various mean and standard deviations; some were less conventional, such as the \emph{persistent entropy} (due to \cite{atienza2020stability} and given by $ h(r) =  \sum^n_{i=1} - \frac{\ell_i}{L} \log_2 \frac{\ell_i}{L},$ where $\ell_i$ is the length of an interval and $L$ is the sum of the length of all intervals),  the \emph{frac 5\%, 100, 200} (the number of intervals whose lengths are shorter than the threshold of 5\% the length of the longest interval, 100ms, and 200ms), or the \emph{signal-to-noise ratio }(which is the ratio of the sum of ``signal'' intervals over the sum of ``noise'' intervals, where the signal is considered to be the intervals longer than the threshold of 5\% the length of the longest interval and those not passing this threshold are considered as ``noise''). 
They introduce new geometric measures called the \emph{topological triangle indices}, which are based on the triangular interpolation on the RR interval histograms classically used in heart rate variability analysis and work by constructing a triangle on the persistence diagram with one side lying on the diagonal, enclosing a set of points with a small percentage of outliers and such that the triangle is as compact as possible. 
With these topological indices, it was shown that optimized combinations of parameters were able to achieve up to 0.84 AUC performance. However, indeed, computing these indices is an involved process.

\begin{table}[!hbt]
\caption{AUC--ROC of linear SVM conducted on quadruples of persistence and non-persistence based features as reported in \cite{graff2021persistent} using three-fold cross-validation and a standard scaler on the input data.}
\label{table:comparison_rank_fns}
\centering
\begin{tabular}{|l|c|}
\hline
Model parameters                                                 & AUC--ROC \\ \hline
Optimal model for frequency parameters                           & 0.75    \\ \hline
Optimal  model for time domain parameters                        & 0.79    \\ \hline
Triangle height, location, no.~of intervals, length sum          & 0.83    \\ \hline
Triangle location, misalignment, no.~of intervals, frac 5\%      & 0.83    \\ \hline
Triangle location, no.~of intervals, length sum, signal to noise & 0.84    \\ \hline
Triangle width, no.~of intervals, length sum, frac 5\%           & 0.84    \\ \hline
\end{tabular}%
\end{table}

For an additional comparison to a typical persistence-based approach, we also trained SVM classifiers on persistence images \citep{adams2017persistence} and persistence landscapes \citep{bubenik2015statistical}---stable vectorizations computed from the barcodes. For persistence images, SVM with a linear kernel achieved optimal performance, with an accuracy of 68.5\% and an AUC--ROC of 0.793, amongst SVM with alternative kernels, sparse SVM, and SVM applied after dimensionality reduction with PCA. For persistence landscapes, SVM with a GRBF kernel achieved optimal performance, with an accuracy of 81.0\% and an AUC--ROC of 0.903. Full results are shown in Appendix \ref{app:PI_HRV}.

In conclusion, the performance we find using rank functions, as a direct and equivalent representation of persistent homology (as opposed to vectorized and manipulated) has better performance over classification with persistence images and is on par and slightly improved over the much more involved approach of computing topological indices from \cite{graff2021persistent} and the one using persistence landscapes.

\section{Rank Functions in Multiparameter Persistent Homology}
\label{sec:mph_applications}

In this section, we explore the use of \emph{multiparameter} persistent homology \citep{carlsson_zomorodian_multiparameter2009} in real data applications.  This is currently an active area of research in TDA, due to interpretive and computational difficulties. 

\subsection{Multiparameter Persistent Homology: The Struggle to Generalize Barcodes} 

There is an important distinction between \emph{single-parameter} persistence modules, i.e.,~diagrams such as those we have considered so far \(M \in \Vect^{(\R,\leq)}\), and \emph{multiparameter} persistence modules \citep{carlsson_theory_2009}, i.e.,~diagrams which allow indexing over the poset \((\R^n,\preceq)\). Here, \(\preceq\) is the product order inherited from the total order in the reals, namely \((x_1,\dots,x_n)\preceq (y_1,\dots,y_n)\) if \(x_i\leq y_i\) for all \(i=1, \dots, n\). The construction discussed in Section \ref{sec:persistence} giving rise to persistent homology can be replicated for these posets to obtain \emph{multifiltrations} and \emph{multiparameter persistent homology}.

The Structure Theorem in Section \ref{sec:pd_and_barcodes} can be extended to general p.f.d.~persistence modules indexed over a small category \citep{botnan_decomposition_2020}, which includes the case of multiparameter persistence. 
As mentioned before, for single-parameter persistence, the only possible indecomposable modules are interval modules, i.e.,~modules \(\mathbb{I} [b,d )\) supported on intervals \([b,d) \subset \R\), allowing for the definition of barcodes as multisets of intervals, as well as the interpretability of births and deaths of topological features corresponding to the intervals.
Although there is a natural extension of the concept of an interval for general posets, the representation type of indecomposable modules over these posets is wider than those supported on intervals, so that no direct, parallel definition of barcode exists.  Moreover, it has been shown that there is in fact no hope for a complete, discrete invariant in multiparameter persistence \citep{carlsson_zomorodian_multiparameter2009}. 

Given the lack of complete invariants for multiparameter persistent homology, a central research interest has been the development of incomplete, interpretable, and computable invariants.
Some strategies to define incomplete invariants include viewing \(n\)-dimensional persistence modules as \(n\)-graded modules over polynomials \citep{carlsson_zomorodian_multiparameter2009} and capitalizing the invariants already existing for such objects, such as minimal presentations and multigraded betti numbers \citep{lesnick_interactive_2015, lesnick_computing_2022} or multigraded associated primes and local cohomology \citep{harrington_stratifying_2019}. Several other proposals bypass the Structure Theorem entirely. \cite{patel_generalized_2018} generalizes the M\"{o}bius inversion in single parameter persistence connecting rank functions and persistence diagrams to define generalized persistence diagrams. \cite{kim_generalized_2021} introduced generalized rank invariants, proving they are the courterpart to generalized persistence diagrams in the M\"{o}bius inversion by \cite{patel_generalized_2018} in the multiparameter setting. Developing a theory of modules over posets, \cite{miller_data_2020} defined QR codes for \(n\)-dimensional modules. Lastly, using resolutions and rank-exact structures, \cite{botnan_signed_2022} defined signed decompositions and signed barcodes, extending single parameter barcodes and including the generalized persistence diagrams by \cite{kim_generalized_2021}. As in single-parameter persistence, a third approach entails vectorizing the output of persistent homology by embedding the modules in a Hilbert space.  Some of these vectorizations are known to result in a loss of information, however. Popular vectorization methods in multiparameter persistence include persistence landscapes \citep{vipond_multiparameter_2020}, images \citep{carriere_multiparameter_2020}, and kernels \citep{corbet_kernel_2019}, among others. 

Remarkably, rank functions can be extended to multiparameter persistence quite naturally. Let
\(\R^{2n+} := \left\{(x,y) \in \left(\{-\infty\} \cup \R\right)^n \times \left(\R \cup \{\infty\}\right)^n: x\preceq y\right\}.\) 

\begin{definition}[Rank Invariant]
\label{def:rank_inv}
Given a p.f.d.\,multiparameter persistence module \(M \in \Vect^{(\R^n, \preceq)}\), its \emph{rank invariant} is defined as
\[\begin{array}{crcl}
    \beta^{M} : &  \R^{2n+} & \to & \Zbb\\
     &  (x,y) & \mapsto & \rank M(x\preceq y)=\dim \Im (M(x\preceq y)).
\end{array}\]
The space of rank invariants for \(n\)-dimensional persistence modules will be denoted by \(\mathcal{I}_n\).
\end{definition}

In any of these approaches, including in the case of rank invariants, applications to real data are in their infancy.  A main obstacle is the lack of efficient software to compute the invariants; the current technology also being restricted to two parameters. Rank invariants for biparameter persistence modules can be computed using RIVET~\citep{rivet,lesnick_interactive_2015,lesnick_computing_2022}---currently the standard software for most strategies in defining multiparameter invariants. Developing efficient algorithms and optimizing existing software remains an active research area~\citep{kerber_efficient_2019, scaramuccia_computing_2020,fugacci_compression_2023}

The question of metrics for rank invariants is equally important as for rank functions.  The well-established \emph{matching distance} for rank invariants restricts multiparameter persistence modules to lines \citep{damico_2003_optimal, damico_using_2006, damico_2010_natural}. 
The matching distance is known to be stable for rank invariants for filtrations obtained as sublevel sets of a function \(f: X \to \R^n\) on \(X\) a triangulable space and with respect to the \(L^\infty\) distance between two filter functions \citep{cerri_betti_2013}.  In a more general setting, the matching distance is also known to be stable with respect to the \emph{interleaving
distance} \citep{lesnick_theory_2015,landinterleaving}; and it is computable in polynomial time \citep{kerber_exact_2019, kerber_fast_2021}.

A significant challenge, however, in using the matching distance in applications---especially in inferential tasks---despite its computability is that it does not induce a Hilbert structure on the space of rank invariants, which is often a condition needed in order to adapt FDA methods (e.g., \cite{crawford2016functional}).
Thus, in our real data application, our focus remains on the \(L^2\) distance on rank invariants, which is also efficiently computed over a discretized grid, providing the necessary Hilbert structure to integrate with FDA methods.


\subsection{Application of Biparameter Rank Functions in Lung Tumor Classification}

We now demonstrate the inferential ability of the biparameter rank functions using nonparametric supervised learning methods on real data. The application focus is to predict lung tumor malignancies from computed tomography (CT) images, which has been studied previously by \cite{VANDAELE2023100657} using single-parameter topological summary statistics. Here, we aim to show that using biparameter persistent homology captures additional distinguishing features of the tumor morphology, both on a local and global scale, which, together with the rank functions, leads to improved classification. 

\noindent
\textbf{Data Description.} We study images from the Lung Image Database Consortium (LIDC), which is freely available from The Cancer Imaging Archive (TCIA) \citep{lidc1,lidc2}. From the LIDC data, we extract a subgroup of 70 chest CT scans, complete with annotations and masks, consisting of those with primary tumors that have either been diagnosed as benign (29) or malignant (41).  

Following the approach in \cite{VANDAELE2023100657}, we convert the collection of CT scan images and masks into 3D point clouds of landmarks on the tumor surfaces by sampling, as shown in Figure \ref{fig:3D_tumour}. On the resulting point clouds, we compute the biparameter rank invariants using two types of bifiltrations, both of which are extensions of the Vietoris--Rips filtration---namely, the \emph{degree--Rips} filtration and the \emph{height--Rips} filtration. The degree refers to the degree of connectivity measured on each vertex of the $1$-skeleton, while the height is measured along the $z$ coordinate, in the direction of stacking of the tumor slices. Using the bifiltration captures prominent features as they develop on the tumor surface along both filtration functions.

\begin{figure*}[htb]
\centering
\begin{minipage}{.35\linewidth}
\centering
\subfloat[]{\label{main:a}\includegraphics[width=\textwidth]{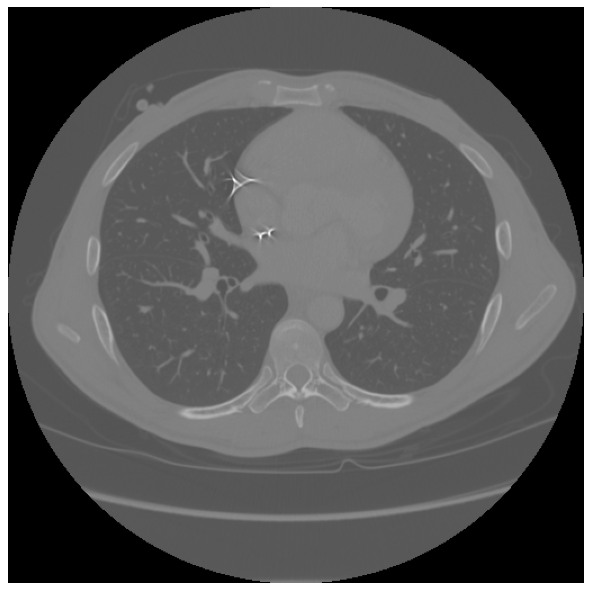}}
\end{minipage}
\begin{minipage}{.15\linewidth}
\centering
\subfloat[]{\label{main:b}\includegraphics[width=\textwidth]{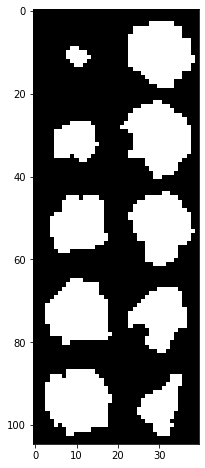}}
\end{minipage}
\begin{minipage}{.45\linewidth}
\centering
\subfloat[]{\label{main:b}\includegraphics[width=\textwidth]{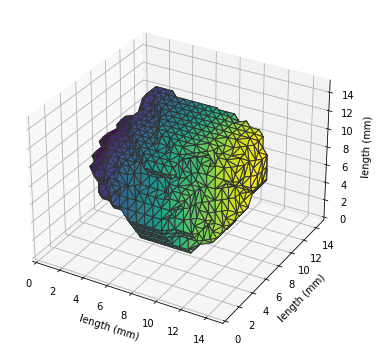}}
\end{minipage}
\caption{An example showing the data extraction process. Annotated CT images (a) from the LIDC combined with masks (b) are converted into a 3D surface (c) from which we can sample a point cloud.}
\label{fig:3D_tumour}
\end{figure*}

\noindent
\textbf{Classification.} We utilize the following two supervised classification methods:
\begin{itemize}
    \item \textbf{$k$-Nearest Neighbors} \citep{knn2}:
    This algorithm is a fundamental classification technique for both multivariate and functional data, where the decision for a new datum is made based on the majority vote of its $k$-closest neighbors. The method is adaptable to general metric spaces since the proximity can be measured using various metrics; the method has been studied in persistent homology by \cite{marchese2017k,leung2022k}. Here, we work with the rank invariants in $L^2$.

    \item \textbf{Functional Maximum Depth} \citep{maxdepth}: This method uses an extended notion of \emph{depth} on functional data to classify curves and surfaces. For a collection of rank invariants, $f_1(x), \dots, f_n(x)$, $x \in \mathcal{X}$, with $\mathcal{X}$ its the domain, we define a \emph{band} as the region or hyperspace bounded by an upper and lower function as follows:
    $$
    B(f_1,\dots,f_n) = \{(x,y): x \in \mathcal{X}, \min_{i=1,\dots, n} f_i(x)\leq y \leq \max_{i=1,\dots, n} f_i(x) \}.
    $$
    The \emph{band depth} $BD$ is the total number of times that $f$ lies within the band formed by a subcollection of the functions
    $BD_{n,J}(f) := \sum^J_{j=2} BD^{(j)}_n (f)$
    for a fixed value $J$, where $2\leq J \leq n$ and
    \begin{equation}
    \label{eq:bd_ind}
    BD^{(j)}_n (f) = {n \choose j}^{-1} \sum_{1\leq i_1 < i_2 <\dots < i_j \leq n} \mathds{1}_{Bd(f)}.
    \end{equation}
Here, $\mathds{1}_{Bd(f)}$ is the indicator function of the set $Bd(f) := \{G(f) \subseteq B(f_{i_1},f_{i_2}, \dots,f_{i_j})\} $ for  $G(f) = \{(x,f(x)):x \in \mathcal{X}\}$.

For our application, we use a modified band depth $MBD$ where instead of using a strict indicator function in \eqref{eq:bd_ind}, we consider the proportion of the hyperspace for which $f$ lies within the band: 
    \begin{equation}
    \label{eq:MBD}
    MBD^{(j)}_n (f) = \\ {n \choose j}^{-1} \sum_{1\leq i_1 <\dots < i_j \leq n} \omega(A(f;\, f_{i_1}, \dots,f_{i_j} ))/\omega(\mathcal{X}),
    \end{equation}
where $\omega$ is a Lebesgue measure on $\mathcal{X}$ and $A(f; f_{i_1} ,f_{i_2}, \dots ,f_{i_j} ) \equiv \{x \in \mathcal{X}: \min_{r=i_1,\dots ,i_j} f_r(x) \leq f(x) \leq \max_{r=i_1,\dots,i_j} f_r(x) \}.$ 
Hence, for any new invariant $f$, it will be assigned to the class in which the modified band depth \eqref{eq:MBD} is maximized. 
\end{itemize}

In the dataset we study, each of the tumor images is classified as either benign or primary malignant.  Our task is to use topological summaries of the images as predictors to determine whether a primary tumor is benign or malignant. We train the classifiers on the biparameter rank invariants computed from the whole dataset. Taking a 75/25 split of the data for training and testing and averaging over 50 iterations, we obtain the results in Table \ref{tab:nocontrast}. Furthermore, 24 of the 29 CT scans of benign tumors and 17 of the 41 CT scans of malignant tumors were taken with added contrast material. Refining to this smaller set, we see further improvements in the predictive accuracies reported in Table \ref{tab:wcontrast}. 

\noindent
\textbf{Results.}
\begin{table}[bth]
\centering
\caption{Accuracy and AUC--ROC of classification between benign and malignant primary tumors in the LIDC dataset.}
\label{tab:nocontrast}
\begin{tblr}{
  column{odd} = {c},
  column{4} = {c},
  column{6} = {c},
  cell{1}{1} = {r=4}{},
  cell{1}{3} = {c=2}{},
  cell{1}{5} = {c=2}{},
  vlines,
  hline{1,5} = {-}{},
  hline{2} = {3-6}{},
  hline{3-4} = {2-6}{},
}
{  Primary\\Benign\\ vs\\Malignant} &             & $k$-NN   &         & MBD    &         \\
                                    &             & Accuracy & AUC–ROC & Accuracy & AUC–ROC \\
                                    & Height–Rips    & 61.4     & 0.599   & 68.8     & 0.691   \\
                                    & Degree–Rips & 63.3     & 0.618   & 70.8     & 0.720   
\end{tblr}
\end{table}
\begin{table}[htb]
\centering
\caption{Accuracy and AUC--ROC of classification between benign and malignant primary tumors in LIDC dataset with added contrast material.}
\label{tab:wcontrast}
\begin{tblr}{
  column{odd} = {c},
  column{4} = {c},
  column{6} = {c},
  cell{1}{1} = {r=4}{},
  cell{1}{3} = {c=2}{},
  cell{1}{5} = {c=2}{},
  vlines,
  hline{1,5} = {-}{},
  hline{2} = {3-6}{},
  hline{3-4} = {2-6}{},
}
{  Primary\\Benign\\ vs\\Malignant\\(with contrast)} &             & $k$-NN   &         & MBD    &         \\
                                                     &             & Accuracy & AUC–ROC & Accuracy & AUC–ROC \\
                                                     & Height–Rips    & 83.8     & 0.830   & 76.9     & 0.768   \\
                                                     & Degree–Rips & 80.0     & 0.791   & 72.5     & 0.727   
\end{tblr}
\end{table}
Without added contrast, our results show that by training a modified maximum depth ($MBD$) classifier, we attain an optimal accuracy and AUC--ROC of 70.8 and 72.0 with the degree--Rips filtration. Overall the performances of MBD classifiers trained on the different bifiltrations are better than the performance of $k$-NN classifiers and also the optimized model in \cite{VANDAELE2023100657} which achieved an AUC--ROC of 67.7 on this dataset.

Moreover, comparing the performance on the subset of data with added contrast material, we find that the $k$-NN classifiers achieved better AUC--ROC with both filtrations than the optimal model in \cite{VANDAELE2023100657} which had an AUC--ROC of 78.0 on average. In fact, the average AUC--ROC for the best $k$-NN classifier based on height--Rips filtration was 83.0. Therefore, indeed we find that the additional information captured by the bifiltration leads to better predictions. 


\section{Discussion}
\label{sec:conclusion}
In this paper, we revisited persistent homology---which provides a geometric representation of data and can be used as a tool for point cloud processing---represented by rank functions in inferential, nonparametric FDA settings. In order to be able to validate our findings from the data analyses, we derived stability conditions on rank functions endowed with an appropriate metric over function space for FDA implementation: namely, the \(L^2\) distance, which provides a Hilbert structure on the space of rank functions. Stability of rank functions, alternatively known as persistent Betti numbers, was well established with respect to the matching distance \citep{cerri_betti_2013}, while, to the best of our knowledge, a thorough understanding of the stability behavior of rank functions endowed with the \(L^p\) metric was previously missing in the literature. We fill this gap with Proposition \ref{prop:stability_truncated}, showing that we can compare to the bottleneck distance for barcodes only when restricting to points away from the diagonal; and Theorem \ref{thm:stability_wasserstein}, where we are also able to find bounds for rank functions with respect to the 1-Wasserstein distance. We also evaluated the performance of the topological representation of data as rank functions in two real-world applications and found that incorporating topological information outperforms previous non-topological methods, as well as other persistence-inspired approaches that use complicated constructions rather than equivalent representations of persistence diagrams.  In addition to performing less well, these topological constructions based on persistent homology are more difficult to interpret and relate back to the original data.  A particularly important contribution in this work that we highlight is in the second application where we used biparameter rank invariants (i.e., rank functions adapted to multiparameter persistent homology).  The adaptation of multiparameter persistent homology to real data is still in its infancy and far from as widespread as in the single-parameter case, because, given the lack of direct extensions for the barcodes and persistence diagrams to higher dimensions, much of the work in the recent years has been foundational and devoted to finding alternative invariants that capture as much information as possible from computing persistent homology. We have found that using rank invariants directly in our real data analysis and machine learning task of classification provides excellent results, encouraging the use of this invariant in multiparameter persistent homology.

This naturally inspires several directions for future research. The first would be extending the theoretical stability results in Section \ref{sec:stability_results} to multiparameter persistent homology represented by rank invariants. 
 As in this work, this would be an important theoretical result needed to validate the experimental findings in this paper as well as justify its continued use in applying multiparameter persistent homology in real data applications.  An additionally important direction to study would be comparative: given the multitude of invariants proposed in the literature on multiparameter persistent homology, understanding the performance of rank invariants in comparison to that of other existing variants would provide a basis and guideline for invariant usage in real data applications.  Specifically, we would like to know whether rank invariants are able to capture more information, as a direct invariant obtained from persistent homology, than other functional vectorizations which embed modules into Hilbert spaces.  In the single-parameter setting, this is true, as we explored in this work.


\section*{Software \& Data Availability}
\label{sec:software_data}
All code and data to reproduce results in this paper are available from \httpsAddr{//github.com/Qiquan-Wang/rank_function_stability.git}. The heart rate variability data were collected and studied by \cite{dgasecki2021} and \cite{knarkiewicz2021}. The lung imaging LIDC data were obtained from TCIA \citep{lidc1}.


\section*{Acknowledgements}

We wish to thank Yueqi Cao, Lorin Crawford, Herbert Edelsbrunner, Claudia Landi, Amit Patel, and Nicole Solomon for helpful discussions.  Q.W.~is grateful to Matthew Williams for his support and guidance.  G.H.P.~and A.M.~also wish to acknowledge the Max-Planck Institute for Mathematics in the Sciences in Leipzig, Germany for hosting their visit in February 2018, where this work began.

I.G.R.~is funded by the UK Engineering and Physical Sciences Research Council [EP/S021590/1] Centre for Doctoral Training in Geometry and Number Theory (The London School of Geometry and Number Theory), University College London.
Q.W.~is funded by a CRUK--Imperial College London Convergence Science PhD studentship at the UK EPSRC Centre for Doctoral Training in Modern Statistics and Machine Learning (2021 cohort, PIs Monod/Williams), which is supported by Cancer Research UK under grant reference [CANTAC721/10021]. A.M.~is supported by the Engineering and Physical Sciences Research Council under grant reference [EP/Y028872/1], a London Mathematical Society Emmy Noether Fellowship [EN-2223-01], and an Imperial College London Elsie Widdowson Fellowship.


\bibliographystyle{apalike} 
\bibliography{main}       

\begin{thebibliography}{}

\bibitem[Adams et~al., 2017]{adams2017persistence}
Adams, H., Emerson, T., Kirby, M., Neville, R., Peterson, C., Shipman, P., Chepushtanova, S., Hanson, E., Motta, F., and Ziegelmeier, L. (2017).
\newblock Persistence images: A stable vector representation of persistent homology.
\newblock {\em The Journal of Machine Learning Research}, 18(1):218--252.

\bibitem[Armato~III et~al., 2011]{lidc1}
Armato~III, S.~G., McLennan, G., Bidaut, L., McNitt-Gray, M.~F., Meyer, C.~R., Reeves, A.~P., Zhao, B., Aberle, D.~R., Henschke, C.~I., Hoffman, E.~A., Kazerooni, E.~A., MacMahon, H., van Beek, E. J.~R., Yankelevitz, D., Biancardi, A.~M., Bland, P.~H., Brown, M.~S., Engelmann, R.~M., Laderach, G.~E., Max, D., Pais, R.~C., Qing, D. P.-Y., Roberts, R.~Y., Smith, A.~R., Starkey, A., Batra, P., Caligiuri, P., Farooqi, A., Gladish, G.~W., Jude, C.~M., Munden, R.~F., Petkovska, I., Quint, L.~E., Schwartz, L.~H., Sundaram, B., Dodd, L.~E., Fenimore, C., Gur, D., Petrick, N., Freymann, J., Kirby, J., Hughes, B., Vande~Casteele, A., Gupte, S., Sallam, M., Heath, M.~D., Kuhn, M.~H., Dharaiya, E., Burns, R., Fryd, D.~S., Salganicoff, M., Anand, V., Shreter, U., Vastagh, S., Croft, B.~Y., and Clarke, L.~P. (2011).
\newblock The lung image database consortium (lidc) and image database resource initiative (idri): A completed reference database of lung nodules on ct scans.
\newblock {\em Medical Physics}, 38(2):915--931.

\bibitem[Armato~III et~al., 2015]{lidc2}
Armato~III, S.~G., McLennan, G., Bidaut, L., McNitt-Gray, M.~F., Meyer, C.~R., Reeves, A.~P., Zhao, B., Aberle, D.~R., Henschke, C.~I., Hoffman, E.~A., Kazerooni, E.~A., MacMahon, H., Van~Beek, E. J.~R., Yankelevitz, D., Biancardi, A.~M., Bland, P.~H., Brown, M.~S., Engelmann, R.~M., Laderach, G.~E., Max, D., Pais, R.~C., Qing, D. P.~Y., Roberts, R.~Y., Smith, A.~R., Starkey, A., Batra, P., Caligiuri, P., Farooqi, A., Gladish, G.~W., Jude, C.~M., Munden, R.~F., Petkovska, I., Quint, L.~E., Schwartz, L.~H., Sundaram, B., Dodd, L.~E., Fenimore, C., Gur, D., Petrick, N., Freymann, J., Kirby, J., Hughes, B., Casteele, A.~V., Gupte, S., Sallam, M., Heath, M.~D., Kuhn, M.~H., Dharaiya, E., Burns, R., Fryd, D.~S., Salganicoff, M., Anand, V., Shreter, U., Vastagh, S., Croft, B.~Y., and Clarke, L.~P. (2015).
\newblock Data from lidc-idri.

\bibitem[Atienza et~al., 2020]{atienza2020stability}
Atienza, N., Gonzalez-D{\'\i}az, R., and Soriano-Trigueros, M. (2020).
\newblock On the stability of persistent entropy and new summary functions for topological data analysis.
\newblock {\em Pattern Recognition}, 107:107509.

\bibitem[Bhattacharya et~al., 2015]{bhattacharya_persistent_2015}
Bhattacharya, S., Ghrist, R., and Kumar, V. (2015).
\newblock Persistent {Homology} for {Path} {Planning} in {Uncertain} {Environments}.
\newblock {\em IEEE Transactions on Robotics}, 31(3):578--590.

\bibitem[Biasotti et~al., 2008a]{Biasotti2008}
Biasotti, S., Cerri, A., Frosini, P., Giorgi, D., and Landi, C. (2008a).
\newblock Multidimensional {S}ize {F}unctions for {S}hape {C}omparison.
\newblock {\em Journal of Mathematical Imaging and Vision}, 32(2):161--179.

\bibitem[Biasotti et~al., 2008b]{Biasotti:2008:DSG:1391729.1391731}
Biasotti, S., De~Floriani, L., Falcidieno, B., Frosini, P., Giorgi, D., Landi, C., Papaleo, L., and Spagnuolo, M. (2008b).
\newblock Describing shapes by geometrical-topological properties of real functions.
\newblock {\em ACM Comput. Surv.}, 40(4):12:1--12:87.

\bibitem[Biwer et~al., 2017]{biwer_windowed_2017}
Biwer, C., Rothberg, A., IglayReger, H., Derksen, H., Burant, C.~F., and Najarian, K. (2017).
\newblock Windowed persistent homology: {A} topological signal processing algorithm applied to clinical obesity data.
\newblock {\em PLOS ONE}, 12(5):e0177696.
\newblock Publisher: Public Library of Science.

\bibitem[Blumberg et~al., 2014]{Blumberg2014}
Blumberg, A.~J., Gal, I., Mandell, M.~A., and Pancia, M. (2014).
\newblock Robust {S}tatistics, {H}ypothesis {T}esting, and {C}onfidence {I}ntervals for {P}ersistent {H}omology on {M}etric {M}easure {S}paces.
\newblock {\em Foundations of Computational Mathematics}, 14(4):745--789.

\bibitem[Boser et~al., 1992]{boser}
Boser, B.~E., Guyon, I.~M., and Vapnik, V.~N. (1992).
\newblock A training algorithm for optimal margin classifiers.
\newblock In {\em Proceedings of the Fifth Annual Workshop on Computational Learning Theory}, COLT '92, page 144–152, New York, NY, USA. Association for Computing Machinery.

\bibitem[Botnan and Crawley-Boevey, 2020]{botnan_decomposition_2020}
Botnan, M. and Crawley-Boevey, W. (2020).
\newblock Decomposition of persistence modules.
\newblock {\em Proceedings of the American Mathematical Society}, 148(11):4581--4596.

\bibitem[Botnan and Hirsch, 2021]{botnan2021consistency}
Botnan, M.~B. and Hirsch, C. (2021).
\newblock On the consistency and asymptotic normality of multiparameter persistent betti numbers.

\bibitem[Botnan et~al., 2022]{botnan_signed_2022}
Botnan, M.~B., Oppermann, S., and Oudot, S. (2022).
\newblock {Signed Barcodes for Multi-Parameter Persistence via Rank Decompositions}.
\newblock In Goaoc, X. and Kerber, M., editors, {\em 38th International Symposium on Computational Geometry (SoCG 2022)}, volume 224 of {\em Leibniz International Proceedings in Informatics (LIPIcs)}, pages 19:1--19:18, Dagstuhl, Germany. Schloss Dagstuhl -- Leibniz-Zentrum f{\"u}r Informatik.

\bibitem[Bramer and Wei, 2020]{bramer_atom-specific_2020}
Bramer, D. and Wei, G.-W. (2020).
\newblock Atom-specific persistent homology and its application to protein flexibility analysis.
\newblock {\em Computational and Mathematical Biophysics}, 8(1):1--35.
\newblock Publisher: De Gruyter Open Access.

\bibitem[Bubenik, 2015]{bubenik2015statistical}
Bubenik, P. (2015).
\newblock Statistical topological data analysis using persistence landscapes.
\newblock {\em The Journal of Machine Learning Research}, 16(1):77--102.

\bibitem[Bubenik et~al., 2015]{bubenik_metrics_2015}
Bubenik, P., de~Silva, V., and Scott, J. (2015).
\newblock Metrics for generalized persistence modules.
\newblock {\em Foundations of Computational Mathematics}, 15(6):1501--1531.
\newblock arXiv:1312.3829 [cs, math].

\bibitem[Bubenik and Scott, 2014]{bubenik_categorification_2014}
Bubenik, P. and Scott, J.~A. (2014).
\newblock Categorification of persistent homology.
\newblock {\em Discrete \& Computational Geometry}, 51(3):600--627.

\bibitem[Burley et~al., 2021]{burley_rcsb_2021}
Burley, S.~K., Bhikadiya, C., Bi, C., Bittrich, S., Chen, L., Crichlow, G.~V., Christie, C.~H., Dalenberg, K., Di~Costanzo, L., Duarte, J.~M., Dutta, S., Feng, Z., Ganesan, S., Goodsell, D.~S., Ghosh, S., Green, R.~K., Guranović, V., Guzenko, D., Hudson, B.~P., Lawson, C., Liang, Y., Lowe, R., Namkoong, H., Peisach, E., Persikova, I., Randle, C., Rose, A., Rose, Y., Sali, A., Segura, J., Sekharan, M., Shao, C., Tao, Y.-P., Voigt, M., Westbrook, J., Young, J.~Y., Zardecki, C., and Zhuravleva, M. (2021).
\newblock {RCSB} {Protein} {Data} {Bank}: powerful new tools for exploring {3D} structures of biological macromolecules for basic and applied research and education in fundamental biology, biomedicine, biotechnology, bioengineering and energy sciences.
\newblock {\em Nucleic Acids Research}, 49(D1):D437--D451.

\bibitem[Cang et~al., 2018]{cang_representability_2018}
Cang, Z., Mu, L., and Wei, G.-W. (2018).
\newblock Representability of algebraic topology for biomolecules in machine learning based scoring and virtual screening.
\newblock {\em PLOS Computational Biology}, 14(1):e1005929.
\newblock Publisher: Public Library of Science.

\bibitem[Cang and Wei, 2017]{cang_analysis_2017}
Cang, Z. and Wei, G.-W. (2017).
\newblock Analysis and prediction of protein folding energy changes upon mutation by element specific persistent homology.
\newblock {\em Bioinformatics}, 33(22):3549--3557.

\bibitem[Cao et~al., 2024]{leung2022k}
Cao, Y., Leung, P., and Monod, A. (2024).
\newblock $k$-{M}eans {C}lustering for {P}ersistent {H}omology.
\newblock {\em Advances in Data Analysis and Classification}, pages 1--25.

\bibitem[Cao and Monod, 2022]{cao2022geometric}
Cao, Y. and Monod, A. (2022).
\newblock A {G}eometric {C}ondition for {U}niqueness of {F}r{\'e}chet {M}eans of {P}ersistence {D}iagrams.
\newblock {\em arXiv preprint arXiv:2207.03943}.

\bibitem[Carlsson and de~Silva, 2010]{carlsson2010zigzag}
Carlsson, G. and de~Silva, V. (2010).
\newblock Zigzag persistence.
\newblock {\em Foundations of computational mathematics}, 10(4):367--405.

\bibitem[Carlsson and Zomorodian, 2007]{carlsson_zomorodian_multiparameter2009}
Carlsson, G. and Zomorodian, A. (2007).
\newblock The theory of multidimensional persistence.
\newblock {\em Discrete and Computational Geometry}, 42:71--93.

\bibitem[Carlsson and Zomorodian, 2009]{carlsson_theory_2009}
Carlsson, G. and Zomorodian, A. (2009).
\newblock The {Theory} of {Multidimensional} {Persistence}.
\newblock {\em Discrete \& Computational Geometry}, 42(1):71--93.

\bibitem[Carri{\`e}re and Blumberg, 2020]{carriere_multiparameter_2020}
Carri{\`e}re, M. and Blumberg, A. (2020).
\newblock Multiparameter {Persistence} {Image} for {Topological} {Machine} {Learning}.
\newblock In {\em Advances in {Neural} {Information} {Processing} {Systems}}, volume~33, pages 22432--22444. Curran Associates, Inc.

\bibitem[Cerri et~al., 2013]{cerri_betti_2013}
Cerri, A., Fabio, B.~D., Ferri, M., Frosini, P., and Landi, C. (2013).
\newblock Betti numbers in multidimensional persistent homology are stable functions.
\newblock {\em Mathematical Methods in the Applied Sciences}, 36(12):1543--1557.

\bibitem[Chazal et~al., 2014]{chazal2014stochastic}
Chazal, F., Fasy, B.~T., Lecci, F., Rinaldo, A., and Wasserman, L. (2014).
\newblock Stochastic convergence of persistence landscapes and silhouettes.
\newblock In {\em Proceedings of the thirtieth annual symposium on Computational geometry}, page 474. ACM.

\bibitem[Cohen-Steiner et~al., 2007]{cohen-steiner_stability_2007}
Cohen-Steiner, D., Edelsbrunner, H., and Harer, J. (2007).
\newblock Stability of {Persistence} {Diagrams}.
\newblock {\em Discrete \& Computational Geometry}, 37(1):103--120.

\bibitem[Corbet et~al., 2019]{corbet_kernel_2019}
Corbet, R., Fugacci, U., Kerber, M., Landi, C., and Wang, B. (2019).
\newblock A kernel for multi-parameter persistent homology.
\newblock {\em Computers \& Graphics: X}, 2:100005.

\bibitem[Cover and Hart, 1967]{knn2}
Cover, T. and Hart, P. (1967).
\newblock Nearest neighbor pattern classification.
\newblock {\em IEEE Transactions on Information Theory}, 13(1):21--27.

\bibitem[Cox and Lee, 2008]{Cox2008PointwiseTW}
Cox, D.~D. and Lee, J.~S. (2008).
\newblock Pointwise testing with functional data using the westfall–young randomization method.
\newblock {\em Biometrika}, 95:621--634.

\bibitem[Crawford et~al., 2020]{crawford2016functional}
Crawford, L., Monod, A., Chen, A.~X., Mukherjee, S., and Rabadán, R. (2020).
\newblock Predicting clinical outcomes in glioblastoma: An application of topological and functional data analysis.
\newblock {\em Journal of the American Statistical Association}, 115(531):1139--1150.

\bibitem[Crawley-Boevey, 2015]{crawley-boevey_decomposition_2015}
Crawley-Boevey, W. (2015).
\newblock Decomposition of pointwise finite-dimensional persistence modules.
\newblock {\em Journal of Algebra and Its Applications}, 14(05):1550066.
\newblock Publisher: World Scientific Publishing Co.

\bibitem[d'Amico et~al., 2003]{damico_2003_optimal}
d'Amico, M., Frosini, P., and Landi, C. (2003).
\newblock Optimal matching between reduced size functions.
\newblock {\em DISMI, Universit`a di Modena e Reggio Emilia}, 35.

\bibitem[d'Amico et~al., 2006]{damico_using_2006}
d'Amico, M., Frosini, P., and Landi, C. (2006).
\newblock Using matching distance in size theory: {A} survey.
\newblock {\em International Journal of Imaging Systems and Technology}, 16(5):154--161.

\bibitem[d'Amico et~al., 2010]{damico_2010_natural}
d'Amico, M., Frosini, P., and Landi, C. (2010).
\newblock {N}atural {P}seudo-{D}istances and {O}ptimal {M}atching between {R}educed {S}ize {F}unctions.
\newblock {\em Acta Applicandae Mathematicae}, 109(2):527--554.

\bibitem[Dauxois et~al., 1982]{dauxois1982asymptotic}
Dauxois, J., Pousse, A., and Romain, Y. (1982).
\newblock Asymptotic theory for the principal component analysis of a vector random function: some applications to statistical inference.
\newblock {\em Journal of multivariate analysis}, 12(1):136--154.

\bibitem[de~Silva and Ghrist, 2007]{de_silva_coverage_2007}
de~Silva, V. and Ghrist, R. (2007).
\newblock Coverage in sensor networks via persistent homology.
\newblock {\em Algebraic \& Geometric Topology}, 7(1):339--358.

\bibitem[Di~Fabio et~al., 2009]{10.1007/978-3-642-04146-4_69}
Di~Fabio, B., Landi, C., and Medri, F. (2009).
\newblock Recognition of {O}ccluded {S}hapes {U}sing {S}ize {F}unctions.
\newblock In Foggia, P., Sansone, C., and Vento, M., editors, {\em Image Analysis and Processing -- ICIAP 2009}, pages 642--651, Berlin, Heidelberg. Springer Berlin Heidelberg.

\bibitem[Duy et~al., 2016]{duy2016limit}
Duy, T.~K., Hiraoka, Y., and Shirai, T. (2016).
\newblock Limit theorems for persistence diagrams.

\bibitem[{Edelsbrunner} et~al., 2002]{edelsbrunner_topological_2002}
{Edelsbrunner}, {Letscher}, and {Zomorodian} (2002).
\newblock Topological {Persistence} and {Simplification}.
\newblock {\em Discrete \& Computational Geometry}, 28(4):511--533.

\bibitem[Edelsbrunner et~al., 2003]{Edelsbrunner03}
Edelsbrunner, H., Letscher, D., and Zomorodian, A. (2003).
\newblock Topological persistence and simplification.

\bibitem[Emmett et~al., 2016]{emmett_multiscale_2016}
Emmett, K., Schweinhart, B., and Rabadan, R. (2016).
\newblock Multiscale {Topology} of {Chromatin} {Folding}.
\newblock In {\em Proceedings of the 9th {EAI} {International} {Conference} on {Bio}-inspired {Information} and {Communications} {Technologies} (formerly {BIONETICS})}, {BICT}'15, pages 177--180, Brussels, BEL. ICST (Institute for Computer Sciences, Social-Informatics and Telecommunications Engineering).

\bibitem[Fasy et~al., 2014]{fasy2014}
Fasy, B.~T., Lecci, F., Rinaldo, A., Wasserman, L., Balakrishnan, S., and Singh, A. (2014).
\newblock Confidence sets for persistence diagrams.
\newblock {\em Ann. Statist.}, 42(6):2301--2339.

\bibitem[Frosini, 1992]{frosini1992measuring}
Frosini, P. (1992).
\newblock Measuring shapes by size functions.
\newblock In {\em Intelligent Robots and Computer Vision X: Algorithms and Techniques}, volume 1607, pages 122--134. International Society for Optics and Photonics.

\bibitem[Frosini and Landi, 1999]{frosini1999size}
Frosini, P. and Landi, C. (1999).
\newblock Size theory as a topological tool for computer vision.
\newblock {\em Pattern Recognition and Image Analysis}, 9(4):596--603.

\bibitem[Frosini and Landi, 2001]{frosini_size_2001}
Frosini, P. and Landi, C. (2001).
\newblock Size {Functions} and {Formal} {Series}.
\newblock {\em Applicable Algebra in Engineering, Communication and Computing}, 12(4):327--349.

\bibitem[Fugacci et~al., 2023]{fugacci_compression_2023}
Fugacci, U., Kerber, M., and Rolle, A. (2023).
\newblock Compression for 2-parameter persistent homology.
\newblock {\em Computational Geometry}, 109:101940.

\bibitem[Gamble and Heo, 2010]{gamble_exploring_2010}
Gamble, J. and Heo, G. (2010).
\newblock Exploring uses of persistent homology for statistical analysis of landmark-based shape data.
\newblock {\em Journal of Multivariate Analysis}, 101(9):2184--2199.

\bibitem[G{\c{a}}secki et~al., 2021]{dgasecki2021}
G{\c{a}}secki, D., Graff, B., Rojek, A., Narkiewicz, K., Graff, G., and Pilarczyk, P. (2021).
\newblock The database of normal rr-intervals of length up to 512 of 41 patients at rest hospitalized due to the episode of acute ischemic stroke.

\bibitem[Gidea, 2017]{gidea_topological_2017}
Gidea, M. (2017).
\newblock Topological {Data} {Analysis} of {Critical} {Transitions} in {Financial} {Networks}.
\newblock In Shmueli, E., Barzel, B., and Puzis, R., editors, {\em 3rd {International} {Winter} {School} and {Conference} on {Network} {Science}}, Springer {Proceedings} in {Complexity}, pages 47--59, Cham. Springer International Publishing.

\bibitem[Graff et~al., 2021]{graff2021persistent}
Graff, G., Graff, B., Pilarczyk, P., Jab{\l}o{\'n}ski, G., G{\k{a}}secki, D., and Narkiewicz, K. (2021).
\newblock Persistent homology as a new method of the assessment of heart rate variability.
\newblock {\em Plos one}, 16(7):e0253851.

\bibitem[H{\"a}berle et~al., 2023]{haberle2023wavelet}
H{\"a}berle, K., Bravi, B., and Monod, A. (2023).
\newblock Wavelet-{B}ased {D}ensity {E}stimation for {P}ersistent {H}omology.
\newblock {\em arXiv preprint arXiv:2305.08999}.

\bibitem[Hamilton et~al., 2022]{hamilton_persistent_2022}
Hamilton, W., Borgert, J.~E., Hamelryck, T., and Marron, J.~S. (2022).
\newblock Persistent {Topology} of {Protein} {Space}.
\newblock In Gasparovic, E., Robins, V., and Turner, K., editors, {\em Research in {Computational} {Topology} 2}, Association for {Women} in {Mathematics} {Series}, pages 223--244. Springer International Publishing, Cham.

\bibitem[Harrington et~al., 2019]{harrington_stratifying_2019}
Harrington, H.~A., Otter, N., Schenck, H., and Tillmann, U. (2019).
\newblock Stratifying {Multiparameter} {Persistent} {Homology}.
\newblock {\em SIAM Journal on Applied Algebra and Geometry}, 3(3):439--471.
\newblock Publisher: Society for Industrial and Applied Mathematics.

\bibitem[Hofer et~al., 2017]{NIPS2017_6761}
Hofer, C., Kwitt, R., Niethammer, M., and Uhl, A. (2017).
\newblock Deep learning with topological signatures.
\newblock In Guyon, I., Luxburg, U.~V., Bengio, S., Wallach, H., Fergus, R., Vishwanathan, S., and Garnett, R., editors, {\em Advances in Neural Information Processing Systems 30}, pages 1634--1644. Curran Associates, Inc.

\bibitem[Kerber et~al., 2019]{kerber_exact_2019}
Kerber, M., Lesnick, M., and Oudot, S. (2019).
\newblock Exact {Computation} of the {Matching} {Distance} on 2-{Parameter} {Persistence} {Modules}.
\newblock page 15 pages.
\newblock Artwork Size: 15 pages Medium: application/pdf Publisher: Schloss Dagstuhl - Leibniz-Zentrum fuer Informatik GmbH, Wadern/Saarbruecken, Germany Version Number: 1.0.

\bibitem[Kerber and Nigmetov, 2019]{kerber_efficient_2019}
Kerber, M. and Nigmetov, A. (2019).
\newblock Efficient approximation of the matching distance for 2-parameter persistence.
\newblock {\em arXiv preprint arXiv:1912.05826}.

\bibitem[Kerber and Rolle, 2021]{kerber_fast_2021}
Kerber, M. and Rolle, A. (2021).
\newblock Fast {Minimal} {Presentations} of {Bi}-graded {Persistence} {Modules}.
\newblock In {\em 2021 {Proceedings} of the {Symposium} on {Algorithm} {Engineering} and {Experiments} ({ALENEX})}, Proceedings, pages 207--220. Society for Industrial and Applied Mathematics.

\bibitem[Kim and Mémoli, 2021]{kim_generalized_2021}
Kim, W. and Mémoli, F. (2021).
\newblock Generalized persistence diagrams for persistence modules over posets.
\newblock {\em Journal of Applied and Computational Topology}, 5(4):533--581.

\bibitem[Krebs and Polonik, 2023]{krebs2023asymptotic}
Krebs, J. and Polonik, W. (2023).
\newblock On the asymptotic normality of persistent betti numbers.

\bibitem[Landi, 2018]{landinterleaving}
Landi, C. (2018).
\newblock {\em The Rank Invariant Stability via Interleavings}, pages 1--10.
\newblock Springer International Publishing, Cham.

\bibitem[Landi and Frosini, 1997]{landi1997new}
Landi, C. and Frosini, P. (1997).
\newblock New pseudodistances for the size function space.
\newblock In {\em Vision Geometry VI}, volume 3168, pages 52--61. International Society for Optics and Photonics.

\bibitem[Landi and Frosini, 2002]{landi2002size}
Landi, C. and Frosini, P. (2002).
\newblock Size functions as complete invariants for image recognition.
\newblock In {\em Vision Geometry XI}, volume 4794, pages 101--110. International Society for Optics and Photonics.

\bibitem[Lees et~al., 2018]{Lees2018HeartRV}
Lees, T., Shad-Kaneez, F., Simpson, A.~M., Nassif, N.~T., Lin, Y., and Lal, S. (2018).
\newblock Heart rate variability as a biomarker for predicting stroke, post-stroke complications and functionality.
\newblock {\em Biomarker Insights}, 13.

\bibitem[Lesnick, 2015]{lesnick_theory_2015}
Lesnick, M. (2015).
\newblock The {Theory} of the {Interleaving} {Distance} on {Multidimensional} {Persistence} {Modules}.
\newblock {\em Foundations of Computational Mathematics}, 15(3):613--650.

\bibitem[Lesnick and Wright, 2015]{lesnick_interactive_2015}
Lesnick, M. and Wright, M. (2015).
\newblock Interactive {Visualization} of 2-{D} {Persistence} {Modules}.
\newblock arXiv:1512.00180 [cs, math] version: 1.

\bibitem[Lesnick and Wright, 2022]{lesnick_computing_2022}
Lesnick, M. and Wright, M. (2022).
\newblock Computing {Minimal} {Presentations} and {Bigraded} {Betti} {Numbers} of 2-{Parameter} {Persistent} {Homology}.
\newblock {\em SIAM Journal on Applied Algebra and Geometry}, 6(2):267--298.
\newblock Publisher: Society for Industrial and Applied Mathematics.

\bibitem[López-Pintado and Romo, 2009]{maxdepth}
López-Pintado, S. and Romo, J. (2009).
\newblock On the concept of depth for functional data.
\newblock {\em Journal of the American Statistical Association}, 104(486):718--734.

\bibitem[Marchese et~al., 2017]{marchese2017k}
Marchese, A., Maroulas, V., and Mike, J. (2017).
\newblock K-means clustering on the space of persistence diagrams.
\newblock In {\em Wavelets and Sparsity XVII}, volume 10394, pages 218--227. SPIE.

\bibitem[Mileyko et~al., 2011]{0266-5611-27-12-124007}
Mileyko, Y., Mukherjee, S., and Harer, J. (2011).
\newblock Probability measures on the space of persistence diagrams.
\newblock {\em Inverse Problems}, 27(12):124007.

\bibitem[Miller, 2020]{miller_data_2020}
Miller, E. (2020).
\newblock Data structures for real multiparameter persistence modules.
\newblock arXiv:1709.08155 [math].

\bibitem[Narkiewicz et~al., 2021]{knarkiewicz2021}
Narkiewicz, K., Graff, B., Graff, G., and Pilarczyk, P. (2021).
\newblock The database of normal rr-intervals of length up to 512 of 46 healthy subjects at rest.

\bibitem[Patel, 2018]{patel_generalized_2018}
Patel, A. (2018).
\newblock Generalized {Persistence} {Diagrams}.
\newblock {\em Journal of Applied and Computational Topology}, 1(3-4):397--419.

\bibitem[Platt, 1998]{platt}
Platt, J. (1998).
\newblock Sequential minimal optimization: A fast algorithm for training support vector machines.
\newblock {\em Advances in Kernel Methods-Support Vector Learning}, 208.

\bibitem[Ramsay, 2002]{ramsayapplied2}
Ramsay, J. O. J.~O. (2002).
\newblock {\em Applied functional data analysis : methods and case studies}.
\newblock Springer, New York, New York.

\bibitem[Ramsay, 2005]{ramsay2}
Ramsay, J. O. J.~O. (2005).
\newblock {\em Functional data analysis}.
\newblock Springer series in statistics. Springer, New York, 2nd ed. edition.

\bibitem[Reininghaus et~al., 2015]{reininghaus2015stable}
Reininghaus, J., Huber, S., Bauer, U., and Kwitt, R. (2015).
\newblock A stable multi-scale kernel for topological machine learning.
\newblock In {\em Proceedings of the IEEE conference on computer vision and pattern recognition}, pages 4741--4748.

\bibitem[Robins and Turner, 2016]{robins_principal_2016}
Robins, V. and Turner, K. (2016).
\newblock Principal {Component} {Analysis} of {Persistent} {Homology} {Rank} {Functions} with case studies of {Spatial} {Point} {Patterns}, {Sphere} {Packing} and {Colloids}.
\newblock {\em Physica D: Nonlinear Phenomena}, 334:99--117.

\bibitem[Rossi and Villa, 2006]{Rossi_2006}
Rossi, F. and Villa, N. (2006).
\newblock Support vector machine for functional data classification.
\newblock {\em Neurocomputing}, 69(7–9):730–742.

\bibitem[Roycraft et~al., 2023]{Roycraft_2023}
Roycraft, B., Krebs, J., and Polonik, W. (2023).
\newblock Bootstrapping persistent betti numbers and other stabilizing statistics.
\newblock {\em The Annals of Statistics}, 51(4).

\bibitem[Saba et~al., 2022]{saba_2022}
Saba, T., Rehman, A., Shahzad, M., Latif, R., Bahaj, S., and Alyami, J. (2022).
\newblock Machine learning for post‐traumatic stress disorder identification utilizing resting‐state functional magnetic resonance imaging.
\newblock {\em Microscopy Research and Technique}, 85.

\bibitem[Scaramuccia et~al., 2020]{scaramuccia_computing_2020}
Scaramuccia, S., Iuricich, F., De~Floriani, L., and Landi, C. (2020).
\newblock Computing multiparameter persistent homology through a discrete morse-based approach.
\newblock {\em Computational Geometry}, 89:101623.

\bibitem[Skraba and Turner, 2021]{skraba_wasserstein_2021}
Skraba, P. and Turner, K. (2021).
\newblock Wasserstein {Stability} for {Persistence} {Diagrams}.
\newblock arXiv:2006.16824 [math].

\bibitem[{The RIVET Developers}, 2020]{rivet}
{The RIVET Developers} (2020).
\newblock Rivet.

\bibitem[Turk and Levoy, 1994]{turk_zippered_1994}
Turk, G. and Levoy, M. (1994).
\newblock Zippered polygon meshes from range images.
\newblock In {\em Proceedings of the 21st Annual Conference on Computer Graphics and Interactive Techniques}, SIGGRAPH '94, page 311–318, New York, NY, USA. Association for Computing Machinery.

\bibitem[Turner et~al., 2014]{Turner2014}
Turner, K., Mileyko, Y., Mukherjee, S., and Harer, J. (2014).
\newblock Fr{\'e}chet {M}eans for {D}istributions of {P}ersistence {D}iagrams.
\newblock {\em Discrete {\&} Computational Geometry}, 52(1):44--70.

\bibitem[Vandaele et~al., 2023]{VANDAELE2023100657}
Vandaele, R., Mukherjee, P., Selby, H.~M., Shah, R.~P., and Gevaert, O. (2023).
\newblock Topological data analysis of thoracic radiographic images shows improved radiomics-based lung tumor histology prediction.
\newblock {\em Patterns}, 4(1):100657.

\bibitem[Vasudevan et~al., 2013]{vasudevan_persistent_2013}
Vasudevan, R., Ames, A., and Bajcsy, R. (2013).
\newblock Persistent homology for automatic determination of human-data based cost of bipedal walking.
\newblock {\em Nonlinear Analysis: Hybrid Systems}, 7(1):101--115.

\bibitem[Verri et~al., 1993]{Verri1993}
Verri, A., Uras, C., Frosini, P., and Ferri, M. (1993).
\newblock On the use of size functions for shape analysis.
\newblock {\em Biological Cybernetics}, 70(2):99--107.

\bibitem[Vietoris, 1927]{vietoris_uber_1927}
Vietoris, L. (1927).
\newblock Über den höheren {Zusammenhang} kompakter {Räume} und eine {Klasse} von zusammenhangstreuen {Abbildungen}.
\newblock {\em Mathematische Annalen}, 97(1):454--472.

\bibitem[Vipond, 2020]{vipond_multiparameter_2020}
Vipond, O. (2020).
\newblock Multiparameter persistence landscapes.
\newblock {\em Journal of Machine Learning Research}, 21(61):1--38.

\bibitem[Xie et~al., 2008]{xie_2008}
Xie, S.-Y., Wang, P.-W., Zhang, H.-J., and Zhao, H.-T. (2008).
\newblock Research on the classification of brain function based on svm.
\newblock In {\em 2008 2nd International Conference on Bioinformatics and Biomedical Engineering}, pages 1931--1934.

\bibitem[Zhu et~al., 2003]{zhu03}
Zhu, J., Rosset, S., Tibshirani, R., and Hastie, T. (2003).
\newblock 1-norm support vector machines.
\newblock In Thrun, S., Saul, L., and Sch\"{o}lkopf, B., editors, {\em Advances in Neural Information Processing Systems}, volume~16. MIT Press.

\bibitem[Zomorodian and Carlsson, 2005]{zomorodian_computing_2005}
Zomorodian, A. and Carlsson, G. (2005).
\newblock Computing {Persistent} {Homology}.
\newblock {\em Discrete \& Computational Geometry}, 33(2):249--274.

\end{thebibliography}


\newpage


\appendix

\section{Proofs}
\label{app:proofs}

We now give the proofs of the key theoretical results presented in Section \ref{sec:stability_results}.

\begin{proof}[Proof of Proposition \ref{prop:stability_truncated}].
Let \(1 \leq p < \infty\) and \(M\) a p.f.d.~persistence module with barcode \(\Barc(M)=\{[b_i, \, d_i): 1 \leq i \leq m\}\). Call \(\delta_i:= d_i-b_i < \infty\) for all \(1\leq i \leq m\), fix \(\delta > 0\), and define \(\eta:= \min\{\delta/2,\, 1,\, \delta_i/2: 1\leq i \leq m\}\). 

Let \(N\) be a p.f.d.~persistence module such that \(d_B(\Dgm(M),\, \Dgm(N)) < \eta\), with barcode \(\Barc(N) = \{[\Tilde{b}_j, \Tilde{d}_j) : 1 \leq j \leq n\}\). By the definition of \(\eta\), the optimal matching \(\phi\) between points in \(\Dgm(M)\) and \(\Dgm(N)\) defined by the bottleneck distance matches all points outside of the diagonal in the diagram \(\Dgm(M)\) to points in \(\Dgm(N)\) outside of the diagonal. In addition, all the remaining points in \(\Dgm(N)\) matched to the diagonal are at an \(\ell^\infty\)-distance of their orthogonal projection to the diagonal of less than \(\delta/2\), which means that \(\beta^N_\delta = 0\) for all of them. 

With these two facts and the additivity of rank functions, we obtain

\begin{align}
        \norm{\beta^M_\delta - \beta^N_\delta}_p & = \norm{\sum_{i = 1}^m \beta_\delta^{\Ibb[b_i, d_i)} - \sum_{j=1}^n\beta_\delta^{\Ibb[\Tilde{b_j}, \Tilde{d_i}) }}_p \nonumber \\
        & = \norm{\sum_{i =1}^m\left( \beta_\delta^{\Ibb[b_i, d_i)} - \beta_\delta^{\Ibb[b_i', d_i')}\right) - \sum_{j \in J}\beta_\delta^{\Ibb[\Tilde{b_j}, \Tilde{d_i}) }}_p \nonumber\\
        & \leq \sum_{i =1}^m \norm{\beta_\delta^{\Ibb[b_i, d_i)} - \beta_\delta^{\Ibb[b_i', d_i') }}_p
    \label{eq:bound_truncated_rf}
\end{align}

where \(\phi (b_i, d_i) = (b_i', d_i')\) for all \(1 \leq i \leq m\) and \(J \subset \{1, \, \dots, \, n\}\) is the subset of indices of points in \(\Dgm(N)\) matched to the diagonal.

We now obtain a bound for \eqref{eq:bound_truncated_rf}. For $i \in \{1,\, \dots, \, m\}$, define the sets 
\begin{align}
    A_i := &\;\{(x,y) \in \mathbb{R}^{2+} : \; b_i \leq x \leq y \leq d_i\}, \nonumber\\
    B_i := &\;\{(x,y) \in \mathbb{R}^{2+} : \;b_i' \leq x \leq y \leq d_i'\}, \nonumber\\
    D_i := &\; A_i \, \Delta \, B_i = (A_i \cup B_i) \setminus (A_i \cap B_i), \label{eq:D_j}
\end{align}
where \(\Delta\) denotes the symmetric difference (see Figure \ref{fig:definition_Dj} for some illustrative examples of \(D_i\)). 

\begin{figure}[htb]
     \centering
     \begin{subfigure}[b]{0.3\columnwidth}
         \centering
         \includegraphics[width=\columnwidth]{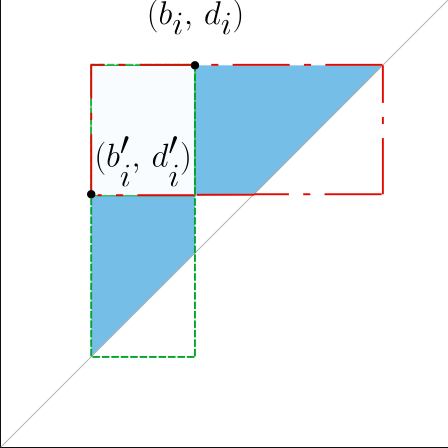}
         \caption{}
         \label{fig:case_A}
     \end{subfigure}
     \hfill
     \begin{subfigure}[b]{0.3\columnwidth}
         \centering
         \includegraphics[width=\columnwidth]{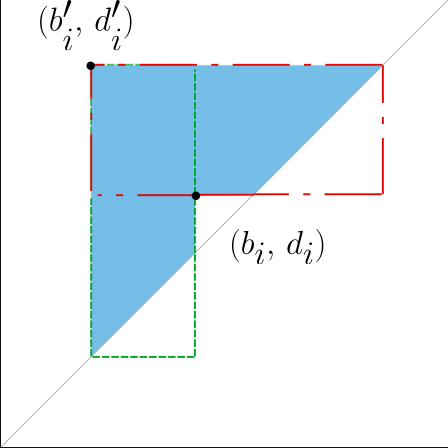}
         \caption{}
         \label{fig:case_B}
     \end{subfigure}
     \hfill
     \begin{subfigure}[b]{0.3\columnwidth}
         \centering
         \includegraphics[width=\columnwidth]{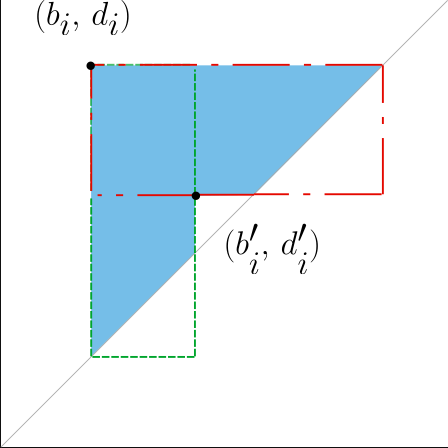}
         \caption{}
         \label{fig:case_C}
     \end{subfigure}
    \caption{Example of domains $D_i$ (shaded in blue) on which rank functions differ and the sketched rectangles (delineated by red dashed lines and green dotted lines) indicate the bound of $\omega(D_i)$ when $M=\mathbb{I} [ b_i,d_i )$ and $N=\mathbb{I} [ b'_i,d'_i )$.}
    \label{fig:definition_Dj}
\end{figure}

Notice that for \((x, y) \in D_i^c\), the truncated rank functions coincide, i.e., \(\beta^{\Ibb [b_i,d_i)} (x,y) = \beta^{\Ibb[b'_i,d'_i )} (x,y)\); also, these rank functions differ by one for \((x,y) \in D_j\). This implies
\begin{equation}
\label{eq:bound_rank_functions_interval_mods}
    \norm{\beta_\delta^{\Ibb[b_i, d_i)} - \beta_\delta^{\Ibb[b_i', d_i') }}_p = \omega(D_i)^{1/p}
\end{equation}
where \(\omega\) denotes the Lebesgue measure in \(\R^2\).

The rectangles depicted in red dashed lines and green dotted lines in Figure \ref{fig:definition_Dj} each have one side of length 
\begin{multline*}
    \max(d'_i, d_i) - \min(b_i, b_i') = \\
    \begin{cases}
    d_i - b_i' \leq d_i-b_i + \norm{(b_i,d_i) - (b_i',d_i')}_\infty & \text{(Figure \ref{fig:case_A}),}\\
    d_i'-b_i \leq d_i-b_i + \norm{(b_i,d_i) - (b_i',d_i')}_\infty & \text{(Figure \ref{fig:case_A}),}\\
    d_i' - b_i' \leq d_i-b_i + 2\norm{(b_i,d_i) - (b_i',d_i')}_\infty  & \textrm{(Figure \ref{fig:case_B}),}\\
    d_i - b_i & \textrm{(Figure \ref{fig:case_C}).}
    \end{cases}
\end{multline*}

The other side of both rectangles is bounded by  $\norm{(b_i,d_i)-(b'_i,d'_i)}_{\infty}$. Observe that this is also a bound for the lengths of the sides of the rectangle in the intersection. 

In case \ref{fig:case_C}, by adding the Lebesgue measure of the rectangles, we get
\[ \omega(D_i) \leq 2\,(d_i - b_i)\, \norm{(b_i,d_i)-(b'_i,d'_i)}_{\infty}\]
where \(\omega(\cdot)\) denotes the Lebesgue measure. In cases \ref{fig:case_A} and \ref{fig:case_B}, adding the Lebesgue measure of the rectangles and triangles that decompose the figures, we obtain
\begin{align*}
    \omega(D_i)  & \leq 2\,(d_i - b_i)\, \norm{(b_i,d_i)-(b'_i,d'_i)}_{\infty}  + 2\, \norm{(b_i,d_i)-(b'_i,d'_i)}_{\infty}^ 2 \\
    & \leq 2\,(d_i - b_i +1)\, \norm{(b_i,d_i)-(b'_i,d'_i)}_{\infty}
\end{align*}
where we have used that $\norm{(b_i,d_i)-(b'_i,d'_i)}_{\infty}\leq 1$ in the last inequality. Setting \(R := \max\{d_i - b_i: 1 \leq i \leq m\}\), we conclude
\begin{equation}
\label{eq:bound_measure_Dj}
    \omega(D_i) \leq 2(R+1) \,  d_B(\Dgm(M), \, \Dgm(N)).
\end{equation}
Returning to~\eqref{eq:bound_truncated_rf}, we obtain the bound
\[ \norm{\beta^M_\delta - \beta^N_\delta}_p \leq m \, (2R + 2)^{1/p} \, d_B(\Dgm(M), \, \Dgm(N))^{1/p}, \]
where \(m\) is the number of points in the persistence diagram of \(M\) and \(R\) is the maximum persistence (see Definition \ref{def:barcode}), so the constant only depends on \(M\), as desired. 
\end{proof}

As previously mentioned in Section \ref{sec:stability_results}, we can always obtain a bounding constant that depends on both modules \(M\) and \(N\) as follows: let \(R'\) be the maximum between the lifetimes of the bars in the barcodes of \(M\) and \(N\), so that \(\omega(D_j) \leq 2 R' \cdot d_B(\Dgm(M),\, \Dgm(N))\) (see \eqref{eq:D_j}); and then use this bound afterwards \eqref{eq:bound_truncated_rf}. In the proof of Proposition \ref{prop:stability_truncated} above, we refine this strategy by bounding with a constant that only depends on the module \(M\).

A natural follow-up question is whether it is possible to achieve a bound such as that in \eqref{eq:bound_measure_Dj}, but with the following dependency with respect to the bottleneck distance:
\begin{equation*}
    \omega(D_j) \leq C(b_j, d_j) \cdot d_B\left(\Dgm(M),\, \Dgm(N)\right)^p
\end{equation*}
for some values \(p\geq 2\).  This would imply a Lipschitz stability condition (notice that for \(p=1\), Proposition \ref{prop:stability_truncated} is actually a Lipschitz condition). In a similar vein to Corollary \ref{cor:L^p_weighted_stability} by \cite{skraba_wasserstein_2021}, the answer to this question is negative, and the key to this fact is the following counterexample. 

\begin{proposition}
\label{prop:no_lipschitz_bottleneck}
    Take \(\mathbb{I} [ b,d)\) and \(C(b,d) = 2(d-b +1) >0\) such that
    \[\omega(D) \leq C(b,d) \cdot \norm{(b,d) - (b',d')}_\infty\]
    for some other interval module \(\mathbb{I} [ b',d' )\), with \(\norm{(b,d) - (b',d')}_\infty < 1\) and \(D\subset \R^{2+}\) defined as
    \begin{align}
    A := &\;\{(x,y) \in \mathbb{R}^{2+} : \; b \leq x \leq y \leq d\}, \nonumber\\
    B := &\;\{(x,y) \in \mathbb{R}^{2+} : \;b' \leq x \leq y \leq d'\}, \nonumber\\
    D := &\; A \, \Delta \, B = (A \cup B) \setminus (A \cap B). \nonumber
\end{align}
    Then
    \[\omega(D) > C(b,d) \cdot \norm{(b,d) - (b',d')}_\infty^p \]
    for every \(p\geq 2\).
\end{proposition}
\begin{proof}
    Let  \(\epsilon > 0\) and consider \(\mathbb{I}_\epsilon: = \mathbb{I} [ b - \epsilon,\, d + \epsilon)\) so that \(\norm{(b,d) - (b-\epsilon,\,d+\epsilon)}_\infty =\epsilon\). In this case we are in a situation similar to the one depicted in Figure \ref{fig:case_B}. Define
    \begin{multline*}
         D_\epsilon := \left\{(x,y) \in \R^{2+}: b\leq x\leq y \leq d\right\} \Delta\\
         \left\{(x,y) \in \R^{2+}: b - \epsilon \leq x\leq y \leq d + \epsilon\right\}.
    \end{multline*}
    For this domain, we exactly have
    \(
    \omega(D_\epsilon) = 2 (d-b) \epsilon + 2 \epsilon^2
    \)
    and taking \(0<\epsilon<1\), we get
    \[
    \omega(D_\epsilon) = 2 (d-b) \epsilon + 2 \epsilon^2 > 2(d-b + 1) \cdot \epsilon^p = C(b,d) \cdot \epsilon^p
    \]
    for every \(p\geq 2\).

    Now, for every other interval module \(\mathbb{I} [ b',d' )\) with \(\norm{(b,d) - (b',d')}_\infty < 1\), we can just consider the previous case with \(\epsilon = \min\left((\abs{b'-b}, \abs{d'-d}\right) < 1\), so that \(D_\epsilon \subset D\) and thus \(\omega(D_\epsilon) \leq \omega(D)\), concluding the proof. 
\end{proof}

From this counterexample and the previous proof, it is not possible to obtain a Lipschitz stability condition for \(p\geq 2\).

\begin{proof}[Proof of Theorem \ref{thm:stability_wasserstein}.]
Let \(M\) be a persistence module with barcode \(\Barc(M)=\{[\tilde{b_i}, \tilde{d_i}) : 1 \leq i \leq m\}\) and $N$ a p.f.d.\,persistence module over $\mathbb{R}$ such that 
\( W_1(\Dgm(M),\,\Dgm(N)) \leq 1\), with barcode \(\Barc(N) = \{[b_j', d'_j): 1 \leq j\leq n\}\). 

Let $\phi$ be the optimal matching between $\Dgm(M)$ and $\Dgm(N)$ induced by the 1-Wasserstein distance. Let \(J \subset \{1,\, \dots,\, n\} \) be the set of subindices corresponding to points in \(\Dgm(N)\) matched to points outside of the diagonal in \(\Dgm(M)\). Call \((b_j, d_j) = \phi^{-1} (b_j', d_j')\), for \(j \in J\), the corresponding matched points in \(\Dgm(M)\). Note that some of the points \((b'_j, d_j') \in \Dgm(N)\) with \(j \in J\) might be on the diagonal.
For $j \in J$, define the set \(D_j\) as in \eqref{eq:D_j}.

From the additivity of the rank function, we obtain the following sequence on inequalities

\begin{align}
        \norm{\beta^M - \beta^N}_p & = \norm{\sum_{i = 1}^m \beta^{\Ibb[\Tilde{b_i}, \tilde{d_i})} - \sum_{j=1}^n\beta^{\Ibb[b_j', d_i') }}_p \nonumber \\
        & = \norm{\sum_{j \in J}\left( \beta^{\Ibb[b_i, d_i)} - \beta^{\Ibb[b_i', d_i')}\right) - \sum_{j \in [m]\setminus J}\beta^{\Ibb[b_j', d_j') }}_p \nonumber\\
        & \leq \sum_{j \in J} \norm{\beta^{\Ibb[b_j, d_j)} - \beta^{\Ibb[b_j', d_j')}}_p + \sum_{j \in [m]\setminus J} \norm{\beta^{\Ibb[b_j', d_j') }}_p
        \label{eq:bound_p_norm_fank_function}
\end{align}
where \([m] = \{1, \, \dots, \, m\}\).

For the first term of the sum \eqref{eq:bound_p_norm_fank_function}, the bound in \eqref{eq:bound_rank_functions_interval_mods} still applies.
Notice also that the \(L^p\) norm of \(\beta^{\mathbb{I}[ b'_j,d'_j )}\) for \(j \in [m]\setminus J\) equals the \(p\)th root of the area of the triangle with vertices at \((b_j',d_j')\), \((b_j', b_j')\), and \((d_j', d_j')\), which is equal to
\(\dfrac{1}{2}\abs{d_j'-b_j'}^2.\) 

For \(p=1\), we can obtain the Lipschitz condition as follows:
\begin{align}
\norm{\beta^M-\beta^N}_{1} &\leq  {\sum_{j \in J}} \, \, \omega (D_j) +  {\sum_{j \in J' \setminus J}} \, \, \frac{1}{2} |d'_j - b'_j|^2 \nonumber\\ 
&\leq  {\sum_{j \in J}} \, \, \omega (D_j) +  {\sum_{j \in J' \setminus J}} \, \, |d'_j - b'_j| \label{eq:ineq_are_triangle_1}\\ 
&\leq 2\, (R+1)\, {\sum_{j \in J}} \, \,   \norm{ (b'_j,d'_j) - (b_j,d_j) }_{\infty} +  {\sum_{j \in J' \setminus J}} \, \, |d'_j - b'_j| \label{eq:1wass_proof_ineq_1}\\ 
&\leq 2\, (R+1)\, {\sum_{j \in J}} \, \,   \left( \abs{b_j -b_j'} + \abs{d_j-d_j'}\right) +  {\sum_{j \in J' \setminus J}} \, \, |d'_j - b'_j| \nonumber\\ 
& \leq 2\, (R + 2) \cdot W_1(\Dgm(M),\,\Dgm(N)), \nonumber
\end{align}
where in \eqref{eq:ineq_are_triangle_1}, we use that \(W_1(\Dgm(M),\,\Dgm(N)) \leq 1\) implies that \(\displaystyle \frac{\abs{d'_j- b'_j}}{2} \leq 1\) for \(j \in J'\setminus J\) and in \eqref{eq:1wass_proof_ineq_1} we have used the bound in \eqref{eq:bound_measure_Dj}.

For the case where \(p=2\), applying the Cauchy--Schwarz inequality we get
\[\sum_{j\in J}\omega(D_j)^{1/2} \leq \left(\abs{J} \sum_{j \in J} \omega(D_j)\right)^{1/2},\]
thus, we obtain
\begin{align}
\norm{\beta^M-\beta^N}_{2} & \leq \left(2\, (R+1)\, \abs{J}\right)^{1/2}\, \left( {\sum_{j \in J}} \, \,   \norm{ (b'_j,d'_j) - (b_j,d_j) }_{\infty}\right)^{1/2} \nonumber \\
 & \qquad +\dfrac{1}{\sqrt{2}}\;{\sum_{j \in J' \setminus J}} \, \, |d'_j - b'_j| \label{eq:ineq_are_triangle}\\  
 & \leq \max\left\{\left(2\, (R+1)\, \abs{J}\right)^{1/2},\, \frac{1}{\sqrt{2}}\right\} \cdot \nonumber \\
 & \left(W_1(\Dgm(M),\,\Dgm(N))^{1/2}  +W_1(\Dgm(M),\,\Dgm(N))\right) \nonumber\\
 & \leq 2 \cdot \max\left\{\left(2\, (R+1)\, \abs{J}\right)^{1/2},\, \frac{1}{\sqrt{2}}\right\}  \cdot\nonumber \\
 & \qquad\qquad\qquad\qquad  W_1(\Dgm(M),\,\Dgm(N))^{1/2} \label{eq:ineq_L2_not_matched}
\end{align}

where in \eqref{eq:ineq_are_triangle} we take the root and use the bound in \eqref{eq:bound_measure_Dj}; and in \eqref{eq:ineq_L2_not_matched}, we use that \(W_1(\Dgm(M),\, \Dgm(N)) \leq 1\).
\end{proof}

\section{Comparison with Persistence Landscapes}
One of the most popular topological vectorization methods in TDA is the \emph{persistence landscape}, proposed by \cite{bubenik2015statistical} as follows. Let \(M\in \Vect^{(\R, \leq)}\) be a single-parameter persistence module and \(\beta^M\) the corresponding rank function. 

\begin{definition}[Persistence landscape, $p$-landscape distance \citep{bubenik2015statistical}]
    For \(k \in \mathbb{N}\) \(k\)th-\emph{persistence landscape} is a function \(\lambda_k: \mathbb{R}\to \mathbb{R}\cup\{-\infty, +\infty\}\) defined as
    \[\lambda_k^M(t) := \sup\{m\geq 0 : \beta^M(t-m, t+m) \geq \}).\]
    Given two different persistence modules, \(M, \, N\in \Vect^{(\R,\leq)}\), the \emph{\(p\)-landscape distance} between them is defined as
\[d_{\lambda, p}(M, N) = \sum_{k=1}^\infty \norm{\lambda^M_k- \lambda^N_k}_p^p.\]
\end{definition}

In \cite{bubenik2015statistical}, several stability results are established for landscapes endowed with this metric. Although landscapes and rank functions are inherently different in nature---where the former is a vectorization of persistence diagrams and barcodes (building from the latter), while the latter is a direct and equivalent representation of diagrams and barcodes---both have been used in real-data applications: a main contribution of this work is the performance assessment of rank functions in inferential machine learning tasks.  This then raises the question of comparison between the stability results associated with landscapes versus those established in this work.

A first observation is that the \(p\)-landscape metric, introduced in \cite{bubenik2015statistical}, involves an infinite sum over the \(L^p\) distances of these landscapes, which is a first distinction from the direct \(L^p\) metrics that we consider over rank functions. Using the \(\infty\)-landscape distance, stability is then achieved with respect to the bottleneck distance between diagrams, which surpasses our Proposition 10 \cite[Theorem 13]{bubenik2015statistical}. However, this is expected, since the persistence landscape is an \emph{incomplete} invariant and thus sacrifices some information encompassed in the persistence diagram for improved stability in the \(L^\infty\) metric, while rank functions, as mentioned previously, are exactly equivalent to persistence diagrams and therefore comprise all topological information of the data captured by persistent homology.

Up until recently, such a stability bound was the best possible, since stability of PH was only rigorously established for the bottleneck distance. However, thanks to new stability results for the $p$-Wasserstein distances established by \cite{skraba_wasserstein_2021}, stability is now possible with respect to these metrics. This is what we achieve in Theorem \ref{thm:stability_wasserstein}. Comparing this result to the \(p\)-landscape stability theorem \cite[Theorem 16]{bubenik2015statistical} is challenging due to different settings and metrics. \cite[Theorem 16]{bubenik2015statistical} considers filtrations over triangulable, compact metric spaces---a restriction we do not impose. In this setting, the \(p\)-landscape metric is compared to the \(L^\infty\) distance between filtering functions in sublevel-set filtrations. Our work extends beyond sublevel-set filtrations, and our \(L^p\) metrics over rank functions are thus not easily comparable to the \(p\)-landscape distances.

\section{HRV Classification Results using Persistence Images and Persistence Landscapes}\label{app:PI_HRV}

We include here the results of the SVM classification using the vectorization techniques of persistence images and persistence landscapes on HRV data. Table \ref{tbl:PI_SVM} shows the average accuracy, AUC--ROC and runtimes (in seconds) of the SVM classifier using persistence images under various kernels with and without dimensionality reduction using PCA. Table \ref{tbl:PL_SVM} shows the same data for the 5 first persistence landscapes \(\lambda_k\), \(1 \leq k \leq 5\). In these tables, the runtime includes: (a) the computation of the PH barcodes for all data, and from them, the computation of the corresponding vectorizations; (b) the training of the corresponding SVM; and (c) the computation of the accuracy and AUC-ROC over five-fold cross-validation. Experiments were run in a processor 11th Gen Intel Core i5-1135G7, with 16GB RAM. Table \ref{tbl:Sparse_SVM} further shows the average accuracy and AUC--ROC of linear support vector classification (LSVC) and sparse LSVC on the data. Recall that where standard LSVC adopts the $L^2$ penalty in the loss function, sparse LSVC adopts the $L^1$ norm, effectively reducing the dimensionality of the feature space \citep{zhu03}.

\begin{table}[htb]
\caption{Average accuracy, AUC--ROC and runtimes of SVM classifiers constructed on persistence images (and persistence images reduced by PCA) computed on HRV data with linear, GRBF, polynomial and sigmoid kernels over ten iterations of five-fold cross-validation.}
\label{tbl:PI_SVM}
\centering
\begin{tabular}{|c|ccc|}
\hline
\multirow{2}{*}{Kernel} & \multicolumn{3}{c|}{SVM}                                                   \\ \cline{2-4} 
                        & \multicolumn{1}{c|}{Accuracy} & \multicolumn{1}{c|}{AUC-ROC} & Runtime (s) \\ \hline
Linear                  & \multicolumn{1}{c|}{65.2}     & \multicolumn{1}{c|}{0.771}   & 4.31        \\ \hline
GRBF                    & \multicolumn{1}{c|}{64.9}     & \multicolumn{1}{c|}{0.756}   & 4.64        \\ \hline
Polynomial (d=2)        & \multicolumn{1}{c|}{48.2}     & \multicolumn{1}{c|}{0.695}   & 5.56        \\ \hline
Polynomial (d=3)        & \multicolumn{1}{c|}{44.9}     & \multicolumn{1}{c|}{0.680}   & 4.05        \\ \hline
Polynomial (d=5)        & \multicolumn{1}{c|}{43.0}     & \multicolumn{1}{c|}{0.678}   & 4.90        \\ \hline
Sigmoid                 & \multicolumn{1}{c|}{60.1}     & \multicolumn{1}{c|}{0.724}   & 4.56        \\ \hline
\multirow{2}{*}{Kernel} & \multicolumn{3}{c|}{PCA + SVM}                                             \\ \cline{2-4} 
                        & \multicolumn{1}{c|}{Accuracy} & \multicolumn{1}{c|}{AUC-ROC} & Runtime (s) \\ \hline
Linear                  & \multicolumn{1}{c|}{65.24}    & \multicolumn{1}{c|}{0.771}   & 1.58        \\ \hline
GRBF                    & \multicolumn{1}{c|}{65.36}    & \multicolumn{1}{c|}{0.756}   & 1.56        \\ \hline
Polynomial (d=2)        & \multicolumn{1}{c|}{44.18}    & \multicolumn{1}{c|}{0.718}   & 1.33        \\ \hline
Polynomial (d=3)        & \multicolumn{1}{c|}{43.02}    & \multicolumn{1}{c|}{0.686}   & 2.01        \\ \hline
Polynomial (d=5)        & \multicolumn{1}{c|}{42.69}    & \multicolumn{1}{c|}{0.676}   & 1.68        \\ \hline
Sigmoid                 & \multicolumn{1}{c|}{61.13}    & \multicolumn{1}{c|}{0.726}   & 2.44        \\ \hline
\end{tabular}%
\end{table}

\begin{table}[]
\caption{Average accuracy, AUC--ROC and runtimes of SVM classifiers constructed on persistence landscapes (and persistence landscapes reduced by PCA) computed on HRV data with linear, GRBF, polynomial and sigmoid kernels over ten iterations of five-fold cross-validation.}
\label{tbl:PL_SVM}
\centering
\begin{tabular}{|c|ccc|}
\hline
\multirow{2}{*}{Kernel} & \multicolumn{3}{c|}{SVM}                                                   \\ \cline{2-4} 
                        & \multicolumn{1}{c|}{Accuracy} & \multicolumn{1}{c|}{AUC-ROC} & Runtime (s) \\ \hline
Linear                  & \multicolumn{1}{c|}{77.0}     & \multicolumn{1}{c|}{0.843}   & 0.421       \\ \hline
GRBF                    & \multicolumn{1}{c|}{81.0}     & \multicolumn{1}{c|}{0.903}   & 0.314       \\ \hline
Polynomial (d=2)        & \multicolumn{1}{c|}{68.2}     & \multicolumn{1}{c|}{0.866}   & 0.427       \\ \hline
Polynomial (d=3)        & \multicolumn{1}{c|}{62.6}     & \multicolumn{1}{c|}{0.852}   & 0.390       \\ \hline
Polynomial (d=5)        & \multicolumn{1}{c|}{58.5}     & \multicolumn{1}{c|}{0.820}   & 0.369       \\ \hline
Sigmoid                 & \multicolumn{1}{c|}{63.2}     & \multicolumn{1}{c|}{0.700}   & 0.632       \\ \hline
\multirow{2}{*}{Kernel} & \multicolumn{3}{c|}{PCA + SVM}                                             \\ \cline{2-4} 
                        & \multicolumn{1}{c|}{Accuracy} & \multicolumn{1}{c|}{AUC-ROC} & Runtime (s) \\ \hline
Linear                  & \multicolumn{1}{c|}{77.02}    & \multicolumn{1}{c|}{0.843}   & 0.524       \\ \hline
GRBF                    & \multicolumn{1}{c|}{80.87}    & \multicolumn{1}{c|}{0.900}   & 0.498       \\ \hline
Polynomial (d=2)        & \multicolumn{1}{c|}{60.48}    & \multicolumn{1}{c|}{0.797}   & 0.479       \\ \hline
Polynomial (d=3)        & \multicolumn{1}{c|}{55.70}    & \multicolumn{1}{c|}{0.868}   & 0.342       \\ \hline
Polynomial (d=5)        & \multicolumn{1}{c|}{46.52}    & \multicolumn{1}{c|}{0.828}   & 0.461       \\ \hline
Sigmoid                 & \multicolumn{1}{c|}{67.47}    & \multicolumn{1}{c|}{0.765}   & 0.417       \\ \hline
\end{tabular}%
\end{table}

\begin{table}[htb]
\caption{Average accuracy and AUC--ROC of linear SVM and sparse linear SVM classifiers on persistence images computed on HRV data over ten iterations of five-fold cross-validation.}
\label{tbl:Sparse_SVM}
\centering
\begin{tabular}{|l|l|l|}
\hline
           & Accuracy & AUC--ROC \\ \hline
Linear SVM & 68.5   & 0.793   \\ \hline
Sparse LSVM & 65.7    & 0.682   \\ \hline
\end{tabular}
\end{table}

\end{document}